\documentclass[12pt,a4paper,oneside]{amsart}
\usepackage[utf8]{inputenc}
\usepackage[margin=3cm]{geometry}
\usepackage{amsmath}
\usepackage{amssymb}
\usepackage{amsthm}
\usepackage[hidelinks]{hyperref}
\usepackage{tikz-cd}
\usepackage{enumitem}

\newtheorem{theorem}{Theorem}[section]
\newtheorem{lemma}[theorem]{Lemma}

\newtheorem{proposition}[theorem]{Proposition}

\DeclareMathOperator{\gal}{Gal}
\DeclareMathOperator{\lcm}{lcm}

\title{Distinct differences of singular moduli}

\author{Guy Fowler}
\address{\parbox{\linewidth}{Department of Mathematics, University of Manchester, M13 9PL, UK\\ Heilbronn Institute for Mathematical Research, Bristol, UK}}
\email{guy.fowler@manchester.ac.uk}

\author{Emanuele Tron}
\address{\parbox{\linewidth}{Institut de mathématiques de Bordeaux, Université de Bordeaux\\351 cours de la Libération, 33405 Talence, France}}
\email{emanuele.tron@math.u-bordeaux.fr}

\subjclass{11G18, 14G35}

\allowdisplaybreaks

\begin{document}
	
	\begin{abstract}
		Let $E_1, E_2 / \mathbb{C}$ be non-isomorphic elliptic curves with complex multiplication. We prove that the pair $(E_1, E_2)$ is characterised, up to isomorphism, by the difference $j(E_1) - j(E_2)$ of the respective $j$-invariants. In other words, we show that if $x_1, x_2, x_3, x_4$ are singular moduli such that $x_1 - x_2 = x_3 - x_4$, then either $(x_1, x_2) = (x_3, x_4)$ or $(x_1, x_3) = (x_2, x_4)$.
	\end{abstract}	
	
	\maketitle
	
	\section{Introduction}
	
	Let $E$ be an elliptic curve over $\mathbb{C}$. Then $E$ is characterised, up to isomorphism, by its $j$-invariant $j(E)$. In this article, we will prove that if $E_1, E_2 / \mathbb{C}$ are non-isomorphic elliptic curves that both have complex multiplication (CM), then the pair $(E_1, E_2)$ is characterised, up to isomorphism, by the difference $j(E_1) - j(E_2)$ of the respective $j$-invariants. The $j$-invariant of a CM elliptic curve is called a singular modulus. Our main result is thus the following statement.
	
	\begin{theorem}\label{thm:main}
	Let $x_1, x_2, x_3, x_4$ be singular moduli. Then
		\[x_1 - x_2 = x_3 - x_4\]
		if and only if either $(x_1, x_2) = (x_3, x_4)$ or $(x_1, x_3) = (x_2, x_4)$.
	\end{theorem}
	
	\medskip
	
	The classical theory of complex multiplication shows that singular moduli are algebraic integers which generate the ring class fields of imaginary quadratic fields. Gross and Zagier \cite{GrossZagier85} proved, under some technical hypotheses, a factorisation formula for the absolute norm of the difference of two singular moduli. Their result shows that such differences are ``highly divisible'', in the sense that the prime factors of the norm are very small compared to the size of the norm. This was generalised to non-fundamental discriminants by Lauter and Viray \cite{LauterViray15}.
	
Known results on singular moduli also suggest that many properties that singular moduli enjoy ought to be shared more generally with differences of two singular moduli. Some examples of interest to us:
	\begin{itemize}
	\item Singular moduli are never units in the ring of algebraic integers \cite[Th.~1.1]{BiluHabeggerKuehne20}, and neither are differences of singular moduli \cite[Cor.~1.3]{Li21}.
	\item Singular moduli are hardly ever $S$-units \cite[Cor.~1.1]{Campagna20}; the same goes for differences of singular moduli \cite[Th.~1]{HerreroMenaresRiveraletelier24}.
	\item Multiplicative relations among singular moduli are rare \cite[Th.~1.2]{PilaTsimerman17}, and similarly so are multiplicative relations among differences of singular moduli \cite[Th.~1.5]{AslanyanEterovicFowler23}, \cite[Cor.~1.5]{FuZhao24}.
	\end{itemize}
	
	\medskip
	
	We identify the modular curve $Y(1) = \mathrm{SL}(2, \mathbb{Z}) \backslash \mathbb{H}$ with the affine line over $\mathbb{C}$ by means of the $j$-invariant. We may then endow $\mathbb{C}^n$ with the structure of a Shimura variety in this way. A point $(x_1, \ldots, x_n) \in \mathbb{C}^n$ such that $x_1, \ldots, x_n$ are all singular moduli is a special point (i.e.~a zero-dimensional special subvariety) of $\mathbb{C}^n$. For a description of the positive-dimensional special subvarieties, see e.g.~\cite[Def.~1.3]{Pila11}. Solutions in singular moduli to polynomial equations therefore correspond to special points of subvarieties of $\mathbb{C}^n$.
	
	The Andr\'e--Oort conjecture for $\mathbb{C}^n$, proved by Pila \cite{Pila11}, states that a subvariety $W \subset \mathbb{C}^n$ contains only finitely many maximal special subvarieties. In particular, a subvariety $W$ contains only finitely many special points outside the union of the positive-dimensional special subvarieties of $W$. Theorem~\ref{thm:main} establishes this in completely explicit form for the subvariety $W \subset \mathbb{C}^4$ given by $W = \mathbb{V}(x - y - w + z)$. Theorem~\ref{thm:main} shows that the two positive-dimensional special subvarieties $\mathbb{V}(x - y, w - z)$ and $\mathbb{V}(x - w, y - z)$ of $W$ contain all the special points of $W$.
	
	Let $\alpha \in \overline{\mathbb{Q}} \setminus \{0\}$. That there are only finitely many special points $(x_1, x_2) \in \mathbb{C}^2$ such that $x_1 - x_2 = \alpha$ is a special case of Andr\'e's theorem \cite{Andre98}.	Pila's proof \cite[Th.~13.2]{Pila11} of the Andr\'e--Oort conjecture for $\mathbb{C}^n$ implies that the number of special points $(x_1, x_2) \in \mathbb{C}^2$ such that $x_1 - x_2 = \alpha$ may be bounded by an ineffective constant which depends only on $[\mathbb{Q}(\alpha) : \mathbb{Q}]$. K\"uhne \cite[Th.~4]{Kuhne13} subsequently proved that this constant could be determined effectively; Edixhoven \cite[Remark~7.1]{Edixhoven98} (cf. \cite[Th.~3.1]{Breuer01}) had previously shown this conditionally under GRH. Scanlon's principle of automatic uniformity, coupled with Pila's proof of Andr\'e--Oort, implies \cite[Th.~4.2]{Scanlon04} that the number of special points $(x_1, x_2) \in \mathbb{C}^2$ such that $x_1 - x_2 = \alpha$ is in fact bounded by an absolute (i.e.~independent of $\alpha$) constant. The same result also follows from \cite[Th.~1.1]{BiluLucaMasser17}; both proofs are ineffective.
	
	A subsequent result of Pila \cite[Th.~1.3]{Pila14a} implies that there are only finitely many special points $(x_1, x_2, x_3, x_4) \in \mathbb{C}^4$ such that $x_1 - x_2 = x_3 - x_4$, but $(x_1, x_2) \neq (x_3, x_4)$ and $(x_1, x_3) \neq (x_2, x_4)$. In other words, there are only finitely many $\alpha \in \overline{\mathbb{Q}}$ for which there exist at least two distinct special points $(x_1, x_2) \in \mathbb{C}^2$ such that $x_1 - x_2 = \alpha$. This result was again ineffective.
	
	Bilu and K\"uhne \cite{BiluKuhne20} proved the Andr\'e--Oort conjecture effectively for linear subvarieties of $\mathbb{C}^n$. In particular, their result \cite[Th.~1.1]{BiluKuhne20} implies that if $(x_1, x_2, x_3, x_4) \in \mathbb{C}^4$ is a special point such that $x_1 - x_2 = x_3 - x_4$, but $(x_1, x_2) \neq (x_3, x_4)$ and $(x_1, x_3) \neq (x_2, x_4)$, then $ \max \left\{ \left\lvert \Delta_i \right\rvert : i=1, 2, 3, 4 \right\} <  10^{88}$. Here $\Delta_i$ denotes the discriminant (defined in Section~\ref{subsec:singmod}) of the singular modulus $x_i$. The constant here is far too large for it to be feasible to deduce Theorem~\ref{thm:main} by a computation (at least with currently available resources). We therefore prove Theorem~\ref{thm:main} directly, without appealing to the main result of \cite{BiluKuhne20}. 
	
	Explicit forms of the Andr\'e--Oort conjecture are known for certain curves in $\mathbb{C}^2$ and surfaces in $\mathbb{C}^3$, e.g.~\cite{AllombertBiluMadariaga15, BiluGunTron22, Fowler23}. Theorem~\ref{thm:main} is, to our knowledge, the first example of such a result for a subvariety of dimension $3$ in $\mathbb{C}^4$.

    \subsection{Overview of proof}

    We now give a brief overview of the proof of Theorem~\ref{thm:main} and an outline of the structure of the paper.

    A result of the first author \cite[Theorem~1.6]{Fowler24} characterises explicitly when three singular moduli are linearly dependent over $\mathbb{Q}$.
    We will apply this to prove Theorem~\ref{thm:main} when either $x_1, x_2, x_3, x_4$ are not pairwise distinct (in Section~\ref{subsec:equal}), or at least one of $x_1, x_2, x_3, x_4$ is rational (in Section~\ref{subsec:rational}).

    Suppose then that $x_1, x_2, x_3, x_4$ are pairwise distinct singular moduli such that
\[ x_1 - x_2 = x_3 - x_4.\]
Let $\Delta_i$ denote the discriminant of $x_i$.
The discriminant $\Delta_i$ is a negative integer.
We may assume that $\lvert \Delta_1 \rvert$ is maximal among the $\lvert \Delta_i \rvert$.
Our aim is to obtain a bound on $\lvert \Delta_1 \rvert$ which is sufficiently sharp as to rule out (after some computations in PARI as necessary) there being any such $x_1, x_2, x_3, x_4$.

    Among the singular moduli of a given discriminant $\Delta$, there is a unique singular modulus of largest absolute value (called the \textit{dominant} singular modulus of that discriminant).
    The singular moduli of a given discriminant $\Delta$ form a complete set of Galois conjugates over $\mathbb{Q}$.
    We recall such background facts in Section~\ref{sec:background}.
    The dominant singular modulus of discriminant $\Delta$ has absolute value $\approx \exp(\pi \lvert \Delta \rvert^{1/2})$, while the absolute value of every other singular modulus of discriminant $\Delta$ is $\ll \exp(\pi \lvert \Delta \rvert^{1/2} / 2)$, see Lemma~\ref{lem:bdsing}.
    
    Applying a suitable Galois automorphism, we may assume that $x_1$ is dominant.
    The proof now naturally divides into four cases, according to how many of $x_2, x_3, x_4$ are also dominant.
    If none of $x_2, x_3, x_4$ is dominant, then it is clear that $\lvert \Delta_1 \rvert$ must be bounded, because $\lvert x_1 \rvert \gg \exp(\pi \lvert \Delta_1 \rvert^{1/2})$, while $\lvert x_i \rvert \ll \exp(\pi \lvert \Delta_1 \rvert^{1/2}/2)$ for $i \geq 2$.
    If exactly one of $x_2, x_3, x_4$ is dominant, then an easy analytic lower bound (Lemma~\ref{lem:small}) for $e^{\pi \sqrt{n}} - e^{\pi \sqrt{n-1}}$ suffices to conclude similarly.
    These two easy cases are handled in Sections~\ref{subsec:0dom} and \ref{subsec:1dom} respectively.

The remaining cases, where at least two of $x_2, x_3, x_4$ are dominant, are more challenging.
In these cases, we would like to find an automorphism $\sigma \in \gal(\overline{\mathbb{Q}} / \mathbb{Q})$ such that $\sigma(x_1) = x_1$ and there are strictly fewer dominant singular moduli among $\sigma(x_2), \sigma(x_3), \sigma(x_4)$ than among $x_2, x_3, x_4$.
We would then be able to reduce to one of the previously-proved cases.

Unfortunately, we are not always able to find such an automorphism $\sigma$; indeed, there are situations where no such $\sigma$ can exist.
Instead, in Section~\ref{subsec:auto}, we use class field theory to prove a result (Proposition~\ref{prop:auto}) which shows that if no such automorphism $\sigma$ exists, then the discriminants $\Delta_1, \Delta_2, \Delta_3, \Delta_4$ must satisfy some very strong restrictions.
In Sections~\ref{sec:2elem} and \ref{sec:fields} we establish some consequences of these restrictions on the discriminants.
 We then use these to show, under these restrictions on the discriminants, that $\lvert \Delta_1 \rvert$ must be bounded in the relevant cases.
The case where two of $x_2, x_3, x_4$ are dominant is handled in Section~\ref{sec:2dom}, and the case where all of $x_2, x_3, x_4$ are dominant is handled in Section~\ref{sec:3dom}.
This completes the proof of Theorem~\ref{thm:main}.

\subsection{Solving equations in singular moduli}

   The reader may wonder to what extent our method could be adapted to handle other linear equations in singular moduli with rational coefficients.
   It would be straightforward to extend our method to handle all equations of the form
   \[ \epsilon_1 x_1 + \epsilon_2 x_2 + \epsilon_3 x_3 + \epsilon_4 x_4 = \alpha,\]
   where $\epsilon_1, \epsilon_2, \epsilon_3, \epsilon_4 \in \{\pm 1\}$ and $\alpha \in \mathbb{Q}$.
   However equations with coefficients of unequal absolute value would pose an obstacle to our approach.
   For example, if $x_1$ and $x_2$ are the only dominant singular moduli occurring and $\lvert \Delta_2 \lvert = \lvert \Delta_1 \rvert - 1$, then our proof obtains a bound on $\lvert \Delta_1 \rvert$ by applying Lemma~\ref{lem:small}.
We could not do this if the coefficient of $x_1$ was strictly smaller than the coefficient of $x_2$ in absolute value, because $e^{\pi \sqrt{n}} \sim e^{\pi \sqrt{n-1}}$ as $n \to \infty$.

Moreover, another fact that necessarily complicates the casework and final answer in such extensions is that plenty of solutions in pairwise distinct singular moduli to linear equations with coefficients of unequal absolute values do exist, unlike in Theorem \ref{thm:main}. 
Examples of such can be generated by variously combining Galois orbits of singular moduli of degree $\le 4$, or by more involved combinations like the solution to
\[ 83\, x_1-49005\, x_2 = 262144\,(x_3+x_4) \]
with $x_1=j\left((-1+\sqrt{-115})/2\right)$, $x_2=j\left( (-1+\sqrt{-75})/2\right)$, $x_3=j\left( (-1+\sqrt{-55})/2\right)$, $x_4=j\left( (-3+\sqrt{-55})/8\right)$.

   Linear equations in more than $4$ variables also seem to require new ideas beyond the methods used here.
   In particular, the arguments in Section~\ref{subsec:unequalare2elem} seem to make essential use of the fact that a linear equation in $4$ singular moduli yields an equality between fields generated by linear combinations of $2$ singular moduli, which then enables us to apply the results on such fields contained in Section~\ref{sec:fields} to this situation.

   On the other hand, the part of our approach that uses class field theory can be readily applied to other situations.
   Indeed, Proposition~\ref{prop:auto} is entirely general and does not even require that the equations being considered are linear.
We therefore expect that it can be used to prove other effective Andr\'e--Oort results.
To prove Proposition~\ref{prop:auto}, we use a result of Cohn \cite[Th.~8.3.12]{Cohn94}. 
Cohn's result and its generalisation by K\"uhne \cite[Cor.~1.2]{Kuhne21} have been applied previously to prove effective Andr\'e--Oort results \cite{BiluKuhne20, Binyamini19}.
In these previous applications, these results from class field theory are used to show that the size of the Galois orbit of a singular modulus $x$ over some particular field $L$ is $\gg c [\mathbb{Q}(x) : \mathbb{Q}]$ for some constant $c > 0$.
This though leads to implied constants which are astronomical in size, even for relatively simple equations, see e.g.~\cite[(40)]{BiluKuhne20}.
In contrast, our Proposition~\ref{prop:auto} shows only that the Galois orbit of $x$ over $L$ is non-trivial, but it does so without introducing such astronomical constants.
It is therefore much better suited to cases where, as in Theorem~\ref{thm:main}, one wants to find all the solutions in singular moduli to a given equation explicitly.
   
    \section{Background}\label{sec:background}
	
	\subsection{Singular moduli}\label{subsec:singmod}
	
The discriminant $\Delta$ of a singular modulus $x$ is defined to be the discriminant of the endomorphism order of an elliptic curve $E / \mathbb{C}$ with $j(E) = x$. In particular, $\Delta$ is a negative integer satisfying $\Delta \equiv 0, 1 \bmod 4$. We may write $\Delta = f^2 D$, where $D$ is the discriminant of the quadratic imaginary field $\mathbb{Q}(\sqrt{\Delta})$ and $f \in \mathbb{Z}_{>0}$. Call $D$ the fundamental discriminant of $x$ and $f$ the conductor of $x$. 
	
	Singular moduli are algebraic integers. The singular moduli of a given discriminant form a complete Galois orbit over $\mathbb{Q}$, see e.g.~\cite[Prop.~13.2]{Cox89}. The number of distinct singular moduli of discriminant $\Delta$ is equal to the class number of $\Delta$, which we denote $h(\Delta)$. 
	
	Let $j \colon \mathbb{H} \to \mathbb{C}$ denote the modular $j$-function, which has the $q$-expansion
	\[ j(z) = \frac{1}{q} + 744 + \sum_{n \geq 1} c(n) q^n,\]
	where $c(n) \in \mathbb{Z}$ for all $n$ and $q = \exp(2 \pi i z)$. The map
	\[ (a, b, c) \mapsto j\left(\frac{-b + \lvert \Delta \rvert^{1/2} i}{2a}\right)\]
	gives a bijection between the set
	\begin{align*} 
	\{\left(a, b, c\right) \in \mathbb{Z}^3 : \, & b^2 - 4ac = \Delta, \gcd\left(a, b, c\right) = 1, \mbox{ and} \\
	& \mbox{ either } -a < b \leq a < c \mbox{ or } 0 \leq b \leq a = c \}
	\end{align*}
	and the set of singular moduli of discriminant $\Delta$. Given a singular modulus $x$, the denominator of $x$ is defined to be the integer $a$ in the corresponding triple $(a, b, c)$. This well-known description of the singular moduli of a given discriminant corresponds to taking, for each singular modulus, its unique preimage (for $j$) in the standard fundamental domain for the action of $\mathrm{SL}(2, \mathbb{Z})$ on $\mathbb{H}$.
	
	For each $\Delta < 0$ such that $\Delta \equiv 0, 1 \bmod 4$, there exists a unique singular modulus with discriminant $\Delta$ and denominator $1$, which is called the dominant singular modulus of discriminant $\Delta$. The dominant singular modulus is given explicitly by
	\[ j\left(\frac{-k + \lvert \Delta \rvert^{1/2} i}{2}\right)\]
	where $k \in \{0, 1\}$ is such that $k \equiv \Delta \bmod 4$. It follows from the $q$-expansion of the $j$-function that $j(z) \in \mathbb{R}$ if $\Re (z) \in \{-1/2, 0\}$. Consequently, the dominant singular modulus of a given discriminant is always real. Moreover, for $x$ the dominant singular modulus of discriminant $\Delta$, one has that $x>0$ if and only if $\Delta \equiv 0 \bmod 4$.
	
	The following bound is also a consequence of the $q$-expansion. Observe that both the lower and upper bounds are monotonically increasing in $\lvert \Delta \rvert$ and monotonically decreasing in $a$. More sophisticated and precise versions are known (see e.g.~\cite[Lemma 2.5]{Pazuki19}), but we only need this much.
	
	\begin{lemma}\label{lem:bdsing}
		Let $x$ be a singular modulus of discriminant $\Delta$ and denominator $a$. Then
		\[ e^{\pi \lvert \Delta \rvert^{1/2}/a} - 2079 \leq \lvert x \rvert \leq e^{\pi \lvert \Delta \rvert^{1/2}/a} + 2079.\]
	\end{lemma}

\begin{proof}
	This is an immediate consequence of \cite[Lemma~1]{BiluMasserZannier13} and the above description of the singular moduli of a given discriminant.
\end{proof}

\begin{lemma}\label{lem:bdfund}
Let $\Delta = f^2 D$ be a discriminant with fundamental discriminant $D$ and conductor $f$. Let $x$ be a singular modulus with fundamental discriminant $D$ and conductor $g$. If $g<f$, then 
\[ \lvert x \rvert \leq e^{-\pi \lvert D \rvert^{1/2}} e^{\pi \lvert \Delta \rvert^{1/2}} + 2079 \leq 0.005 e^{\pi \lvert \Delta \rvert^{1/2}} + 2079.\]
\end{lemma}

\begin{proof}
    We follow the argument in \cite[\S4.2]{FayeRiffaut18}. Observe that
    \[e^{\pi g \lvert D \rvert^{1/2}} \leq e^{\pi (f - 1)\lvert D \rvert^{1/2}} \leq e^{-\pi \lvert D \rvert^{1/2}} e^{\pi \lvert \Delta \rvert^{1/2}} \leq e^{-\pi \sqrt{3}} e^{\pi \lvert \Delta \rvert^{1/2}}, \]
since $\lvert D \rvert \geq 3$. The result then follows from Lemma~\ref{lem:bdsing}, since $e^{-\pi \sqrt{3}} = 0.0043\ldots$
\end{proof}

\begin{proposition}[{\cite[Lemmata~5.1 \& 5.3]{BiluLucaMasser17}}]\label{prop:dombig}
Let $x_1, x_2$ be distinct singular moduli of respective discriminants $\Delta_1, \Delta_2$. Suppose that $\lvert \Delta_2 \rvert \leq \lvert \Delta_1 \rvert$. If $x_1$ is dominant, then $\lvert x_2 \rvert < \lvert x_1 \rvert$.
\end{proposition}

	\subsection{An analytic estimate}

	\begin{lemma}\label{lem:small} Let $d \in \mathbb{R}$. If $d \geq 5$, then
		\[ e^{\pi \sqrt{d}}-e^{\pi \sqrt{d-1}} > e^{\pi \sqrt{d}}/\sqrt{d}. \]	
	\end{lemma}
	\begin{proof}
	For $x>1$, let
		\[ f(x) = \sqrt{x} \left(1 - \exp(\pi (\sqrt{x-1} - \sqrt{x}))\right).\]
		Note that
		\[ \sqrt{x-1} - \sqrt{x} = \frac{-1}{\sqrt{x-1} + \sqrt{x}},\]
		so
		\[ f(x) = \sqrt{x} \left(1 - \exp\left(\frac{-\pi}{\sqrt{x-1} + \sqrt{x}}\right)\right)  \geq \sqrt{x}\left(1 - \exp\left(\frac{-\pi}{2 \sqrt{x}}\right)\right).\]
		Recall that
		\[ e^{-y} < 1 - y +\frac{y^2}{2}\]
		for every $y>0$. Hence, for $x> 1$,
		\[ f(x) \geq \frac{\pi}{2} - \frac{\pi^2}{8 \sqrt{x}}.\]
		In particular, if $x \geq 5$, then $f(x) > 1$ as required.
	\end{proof}

	\section{Some class field theory}\label{sec:class}

    \subsection{Automorphisms of fields generated by singular moduli}\label{subsec:auto}

Let $D_1, D_2$ be distinct fundamental discriminants and let $x$ be a singular modulus of discriminant $\Delta = f^2 D_1$ for some $f \in \mathbb{Z}_{>0}$. 
The main result of this section (Proposition~\ref{prop:auto}) will show that there exists an automorphism $\sigma \in \gal(\overline{\mathbb{Q}} / \mathbb{Q})$ such that $\sigma(x) \neq x$, but $\sigma(y) = y$ for every singular modulus $y$ of fundamental discriminant $D_2$, unless $\Delta$ is a $2$-elementary discriminant (defined below).
Proposition~\ref{prop:auto} will be vital to the proof of Theorem~\ref{thm:main}, because it allows us to reduce the number of dominant singular moduli of distinct fundamental discriminants that occur in an equation
\[ x_1 - x_2 = x_3 - x_4.\]
In Section~\ref{sec:2elem}, we will then prove some characterisations of $2$-elementary discriminants, which will enable us to handle those cases in the proof of Theorem~\ref{thm:main} where no such automorphism $\sigma$ exists.

	Let $K$ be an imaginary quadratic field. Write $D$ for the discriminant of $K$. Let $f \in \mathbb{Z}_{>0}$. Define the ring class field of $K$ of conductor $f$, which is denoted by $K[f]$, to be equal to $K(x)$, where $x$ is any singular modulus of discriminant $\Delta = f^2 D$. This definition is independent of the choice of $x$, by e.g.~\cite[Th.~11.1]{Cox89}. 
    The ring class field $K[f]$ is an abelian extension of $K$ and a Galois extension of $\mathbb{Q}$. The Galois group $\gal(K[f]/K)$ is isomorphic to the class group of discriminant $\Delta$, denoted $\mathrm{cl}(\Delta)$. See \cite[p.~180]{Cox89}.
    We now recall some terminology from \cite{BiluGunTron22}.

A group $G$ is called $2$-elementary if every element of $G$ has order $\leq 2$. Note that a $2$-elementary group is always abelian and a finite $2$-elementary group is isomorphic to $(\mathbb{Z} / 2\mathbb{Z})^n$ for some $n \in \mathbb{Z}_{\geq 0}$. 

A finite field extension $L/K$ is called $2$-elementary if the extension $L/K$ is Galois and the Galois group $\gal(L/K)$ is $2$-elementary. A number field $L$ is called $2$-elementary if the extension $L/ \mathbb{Q}$ is $2$-elementary.

A discriminant $\Delta=f^2 D$ is called $2$-elementary if the finite field extension $K[f]/K$ is $2$-elementary, where $K = \mathbb{Q}(\sqrt{D})$. 

\begin{proposition}\label{prop:auto}
    Let $\Delta = f^2 D_1$ be a discriminant, where $D_1$ is the corresponding fundamental discriminant. Suppose that $\Delta$ is not $2$-elementary. Let $D_2$ be a fundamental discriminant such that $D_2 \neq D_1$. Then there exists $\sigma \in \gal(\overline{\mathbb{Q}} / \mathbb{Q})$ such that:
    \begin{enumerate}
        \item if $x$ is a singular modulus of discriminant $\Delta$, then $\sigma(x) \neq x$.
        \item if $y$ is a singular modulus with fundamental discriminant $D_2$, then $\sigma(y) = y$.
    \end{enumerate}
\end{proposition}

Note that some condition on the discriminant $\Delta$ is necessary in Proposition~\ref{prop:auto}, because there do exist singular moduli $x, y$ with distinct fundamental discriminants for which $\mathbb{Q}(x) = \mathbb{Q}(y)$, see \cite[Table~2]{AllombertBiluMadariaga15} for a list; the discriminants of such $x, y$ are $2$-elementary.

To prove Proposition~\ref{prop:auto}, we will use the following result of Cohn \cite{Cohn94}.
Before stating his result, we give a definition.
The transfer field of an imaginary quadratic field $K$, denoted by $K^\mathrm{tf}$, is defined by
\[ K^{\mathrm{tf}} = \bigcup_{f \in \mathbb{Z}_{>0}} K[f].\]
The transfer field $K^\mathrm{tf}$ is an abelian extension of $K$ and a Galois extension of $\mathbb{Q}$.
In particular, if $x$ is singular modulus of discriminant $\Delta$ such that $\mathbb{Q}(\sqrt{\Delta}) = K$, then  $x \in K^\mathrm{tf}$.

\begin{proposition}[{\cite[Th. 8.3.12]{Cohn94}}, cf. {\cite[Cor.~1.2]{Kuhne21}}]\label{prop:Cohn}
	Let $K_1 \neq K_2$ be distinct, imaginary quadratic fields. Then the group
	\[ \gal(K_1^\mathrm{tf} \cap K_1 K_2^\mathrm{tf} / K_1) \]
is $2$-elementary.
\end{proposition}

K\"uhne \cite{Kuhne21} proved a generalisation of Cohn's result. Cohn's result was previously applied by Binyamini \cite[\S2.3]{Binyamini19} to obtain some effective Andr\'e--Oort results.

Before proving Proposition~\ref{prop:auto}, we first establish some auxiliary results.

\begin{lemma}\label{lem:classfield1}
	Let $K_1 \neq K_2$ be distinct, imaginary quadratic fields with respective discriminants $D_1, D_2$. Suppose that $f \in \mathbb{Z}_{>0}$ is such that
	\[ K_1[f] \subset K_1 K_2^\mathrm{tf}.\]
	Then the discriminant $f^2 D_1$ is $2$-elementary.
\end{lemma}

\begin{proof}
	By assumption, $K_1[f] \subset K_1 K_2^\mathrm{tf}$. Hence,
	\[ K_1[f] \subseteq K_1^\mathrm{tf} \cap K_1 K_2^\mathrm{tf}.\]
So $\gal(K_1[f]/K_1)$ is isomorphic to a quotient of the group $\gal(K_1^\mathrm{tf} \cap K_1 K_2^\mathrm{tf} / K_1)$. By Proposition~\ref{prop:Cohn}, the group $\gal(K_1^\mathrm{tf} \cap K_1 K_2^\mathrm{tf} / K_1)$ is $2$-elementary. Hence, the group $\gal(K_1[f]/K_1)$ is also $2$-elementary, i.e.~the discriminant $f^2 D_1$ is $2$-elementary. 
\end{proof}

We recall the following fact from Galois theory.

\begin{proposition}[{\cite[Th.~VI.1.12]{Lang02}}]\label{prop:Lang}
	Let $L / k$ be a Galois extension and let $F / k$ be an arbitrary extension. Then $LF / F$ and $L / L \cap F$ are both Galois extensions and the map $\sigma \mapsto \sigma |_L$ is an isomorphism $\gal(LF/ F) \xrightarrow{\sim} \gal(L / L \cap F)$
\end{proposition} 

We need one more lemma.

\begin{lemma}\label{lem:classfield2}
	Let $K_1 \neq K_2$ be distinct, imaginary quadratic fields with respective discriminants $D_1, D_2$. Suppose that $f \in \mathbb{Z}_{>0}$ is such that the discriminant $f^2 D_1$ is not $2$-elementary. Then $K_1[f] K_2^\mathrm{tf} / K_1 K_2^\mathrm{tf}$ is a proper Galois extension.
\end{lemma}

\begin{proof}
	The extension $K_1[f]/K_1$ is Galois. Hence, by Proposition~\ref{prop:Lang}, the extensions $K_1[f] K_2^\mathrm{tf} / K_1 K_2^\mathrm{tf}$ and $K_1[f] / K_1[f] \cap K_1 K_2^\mathrm{tf}$ are both Galois and have isomorphic Galois groups. 
	
	By assumption, the discriminant $f^2 D_1$ is not $2$-elementary. Hence, by Lemma~\ref{lem:classfield1}, we have that $K_1[f] \not \subset K_1 K_2^\mathrm{tf}$. So $K_1[f] \supsetneq K_1[f] \cap K_1 K_2^\mathrm{tf}$. Hence, the group $\gal(K_1[f] / K_1[f] \cap K_1 K_2^\mathrm{tf})$ is non-trivial, and so the group $\gal(K_1[f] K_2^\mathrm{tf} / K_1 K_2^\mathrm{tf})$ is also non-trivial. In particular, $K_1[f] K_2^\mathrm{tf} / K_1 K_2^\mathrm{tf}$ is a proper Galois extension.
\end{proof}

We are now in a position to deduce Proposition~\ref{prop:auto}.

\begin{proof}[Proof of Proposition~\ref{prop:auto}]
    For $i=1, 2$, let $K_i = \mathbb{Q}(\sqrt{D_i})$. Then $K_1 \neq K_2$. Since $\Delta = f^2 D_1$ is not $2$-elementary, the extension $K_1[f] K_2^\mathrm{tf} / K_1 K_2^\mathrm{tf}$ is a proper Galois extension by Lemma~\ref{lem:classfield2}. Hence there exists a non-identity element $\sigma \in \gal(K_1[f] K_2^\mathrm{tf} / K_1 K_2^\mathrm{tf})$. In particular, $\sigma$ must act non-trivially on $K_1[f]$. If $x$ is a singular modulus of discriminant $\Delta$, then $K_1(x) = K_1[f]$, and so we must have that $\sigma(x) \neq x$. If $y$ is a singular modulus with fundamental discriminant $D_2$, then $y \in K_2^\mathrm{tf}$ and hence $\sigma(y) = y$. The proposition thus follows by taking an extension of $\sigma$ to an automorphism of $\overline{\mathbb Q}$. 
\end{proof}

    \subsection{Characterising \texorpdfstring{$2$}{2}-elementary discriminants}\label{sec:2elem}

    We will now show that $2$-elementary discriminants are of a particularly special form.
We will use these facts later in Sections~\ref{sec:2dom} and \ref{sec:3dom}.
    
        Tatuzawa \cite[Th.~2]{Tatuzawa51} proved that Siegel's \cite[(1)]{Siegel35} classical, and classically ineffective, lower bound for the class number of imaginary quadratic fields can be made effective, apart from possibly a single exceptional imaginary quadratic field. The following result about $2$-elementary fundamental discriminants is an explicit consequence of his theorem.

\begin{proposition}\label{prop:Tat}
There exists a fundamental discriminant $D_*$ with the following properties:
\begin{enumerate}
    \item $h(D_*) \geq 128$,
    \item if $\Delta = f^2 D$ is a $2$-elementary discriminant, then either $D = D_*$ or $h(\Delta) \leq 16$.
\end{enumerate}
In particular, $\lvert D_* \rvert > 8000$. If $\Delta$ is a $2$-elementary discriminant with $h(\Delta) \leq 16$, then $\lvert \Delta \rvert \leq 7392$.
\end{proposition}

\begin{proof}
    The first two properties are \cite[Prop.~2.11]{BiluGunTron22}. The rest is a simple calculation in PARI. One may verify that if $D$ is a fundamental discriminant such that $\lvert D \rvert \leq 8000$, then $h(D) \leq 120$. Similarly, one checks that every $2$-elementary discriminant $\Delta$ with $h(\Delta) \leq 16$ satisfies $\lvert \Delta \rvert \leq 7392$. For the latter computation, we use the fact \cite[Prop.~7.3]{Fowler24} that $h(\Delta) \leq 32$ implies that $\lvert \Delta \rvert \leq 166147$.
\end{proof}

Note that there may not exist any $2$-elementary discriminants $\Delta$ with $h(\Delta) > 16$. In this case, any fundamental discriminant $D_*$ with $h(D_*) \geq 128$ satisfies Proposition~\ref{prop:Tat}; for definiteness of notation, we set $D_*$ to be the least (in absolute value) such fundamental discriminant, so $D_* = -8399$ . If there do exist $2$-elementary discriminants $\Delta$ with $h(\Delta) > 16$, then $D_*$ is uniquely determined.

We will now show that the conductors of $2$-elementary discriminants are bounded. Lemma~\ref{lem:bdfund} and Proposition~\ref{prop:2elemcond} together imply a bound on the absolute value of a singular modulus $x$ with $2$-elementary discriminant $\Delta = f^2 D$ solely in terms of the fundamental discriminant $D$.

\begin{proposition}\label{prop:2elemcond}
	Let $\Delta = f^2 D$ be a $2$-elementary discriminant. Then $f \leq 60$ and $f \mid 2^3 \cdot 3 \cdot 5 \cdot 7$. If $D \neq D_*$, then $f \leq 8$.
\end{proposition}

This is a strengthening of a special case of \cite[Propositions~2.9 \& 2.10]{BiluGunTron22}, which imply that any such conductor $f$ satisfies $f \leq 28560$. 
The proof requires a couple of easy lemmata.
The second is an immediate consequence of the first, since quotients of $2$-elementary groups are themselves $2$-elementary.

    \begin{lemma}[{\cite[Prop.~3.1]{AllombertBiluMadariaga15}}]\label{lem:compring}
		Let $K$ be an imaginary quadratic field. Write $D$ for the discriminant of $K$ and let $f, g \in \mathbb{Z}_{>0}$. Then
		\[ K[f] K[g] \subseteq K\left[\lcm\left(f, g\right)\right].\]
	\end{lemma}

\begin{lemma}\label{lem:2elemsub}
	Let $\Delta = f^2 D$ be a $2$-elementary discriminant and set $K = \mathbb{Q}(\sqrt{D})$. If $g \in \mathbb{Z}_{>0}$ is such that $K[g] \subseteq K[f]$, then the discriminant $g^2 D$ is $2$-elementary. In particular, the discriminant $D$ is $2$-elementary (since $K[1] \subseteq K[f]$ by Lemma~\ref{lem:compring}).
\end{lemma}

Now we may prove Proposition~\ref{prop:2elemcond}.

\begin{proof}[Proof of Proposition~\ref{prop:2elemcond}]
Let $\Delta = f^2 D$ be a $2$-elementary discriminant. If $D \neq D_*$, then $\lvert \Delta \rvert \leq 7392$ by Proposition~\ref{prop:Tat}. We may then verify in PARI that $f \leq 8$ by checking all such $\Delta$. So we may assume subsequently that $D = D_*$. In particular, $D \neq -3, -4$.
	
	The proof now follows the approach in \cite[Prop.~2.9]{BiluGunTron22}. Since $\Delta$ is $2$-elementary, $D$ is $2$-elementary by Lemma~\ref{lem:2elemsub}. Therefore, by \cite[(2.9)]{BiluGunTron22}, we have that
	\[ h(\Delta) = 2^{\rho_2(\Delta)} \mbox{ and } h(D) = 2^{\rho_2(D)},\]
	where $\rho_2(\Delta)$ denotes the dimension of $\mathrm{cl}(\Delta) / \mathrm{cl}(\Delta)^2$ as an $\mathbb{F}_2$-vector space and analogously for $\rho_2(D)$. By \cite[Prop.~2.3]{BiluGunTron22}, $\rho_2(D) = \omega(D) -1$ and $\rho_2(\Delta) \leq \omega(\Delta)$, where $\omega(\cdot)$ denotes the number of distinct prime divisors function. Therefore,
	\[ \rho_2(\Delta) - \rho_2(D) \leq \omega(\Delta) - \omega(D) + 1 \leq \omega(f) + 1.\]
	
	Let
	\[ \Psi = f \prod_{p \mid f} \left(1 - \frac{\left(D/p\right)}{p}\right),\]
	where $(D/p)$ denotes the Kronecker symbol. From the class number formula \cite[(2.4) \& (2.5)]{BiluGunTron22}, we have that
	\[ h(\Delta) = h(D) \Psi.\]
	Hence,
	\[ \Psi = 2^{\rho_2(\Delta) - \rho_2(D)}.\]
	So $\Psi$ is a power of $2$ and $\Psi \mid 2^{\omega(f) + 1}$. This certainly implies that \cite[(2.12)]{BiluGunTron22} holds.  We may thus argue as in the proof of \cite[Prop.~2.9]{BiluGunTron22} to obtain \cite[(2.15)]{BiluGunTron22}, i.e.
	\[ f = 2^k p_1 \cdots p_m\]
	for some non-negative integers $k, m$ and distinct odd primes $p_1 < \ldots < p_m$ not dividing $D$. Note that $k=0$ and $m=0$ are both allowed. Further, for each $p_i$, either $p_i - 1$ or $p_i + 1$ divides $\Psi$ (and hence is a power of $2$).
	
	First, we show that $k \leq 3$. Suppose that $k \geq 1$. Then $\omega(f) = m + 1$. Hence,
	\[ 2^{m+2} = 2^{\omega(f)+1} \geq \Psi \geq 2^{k-1} (p_1 - 1) \cdots (p_m - 1) \geq 2^{m + k - 1}.\]
	Therefore, $m + k - 1 \leq m + 2$, and so $k \leq 3$. 
	
	Suppose that $\omega(f) = 1$. So $\Psi \mid 4$ and either $f = 2^k$ for some $k \leq 3$ or $f = p_1$ for some odd prime $p_1$ such that either $p_1 - 1$ or $p_1 + 1$ divides $\Psi$. Thus, if $f = p_1$, then $p_1 \in \{3, 5\}$. We are thus done in this case.
	
	So we may assume that $\omega(f) \geq 2$. So $m \geq 1$. Suppose that $p_m > 7$. One of $p_m - 1$ and $p_m + 1$ is a power of $2$, so $p_m \geq 17$. Therefore,
	\[ 2^{m+2} \geq 2^{\omega(f) + 1} \geq \Psi \geq (p_1 - 1) \cdots (p_m - 1).\]
	If $m \geq 2$, then
	\[ 2^{m+2} \geq (3 - 1) (5 - 1)^{m-2} 16 = 2^{1 + 2m}\]
	and so $1 + 2m \leq m + 2$, which is false since $m \geq 2$. So we must have that $m = 1$, and hence $\omega(f) = 2$. Hence, $f = 2^k p_m$ for some $k \geq 1$. We thus have that
	\[ 8 = 2^{\omega(f) + 1} \geq \Psi \geq 2^{k-1} (p_m - 1) \geq 16,\]
	which is absurd. So $p_m \leq 7$. 
	
	So $p_1, \ldots, p_m \in \{3, 5, 7\}$. In particular, $m \leq 3$ and so $\omega(f) \leq 4$. Suppose that $\omega(f) = 4$. Then $f = 2^k \cdot 3 \cdot 5 \cdot 7$ for some $k \geq 1$. Thus
	\[ 2^5 = 2^{\omega(f)+1} \geq \Psi \geq 2^{k-1} (3-1)(5-1)(7-1) \geq 48,\]
	which is false. So $\omega(f) \leq 3$.
	
	Suppose that $\omega(f) = 3$. So $\Psi \leq 2^4$. If $k = 0$, then $f = 3 \cdot 5 \cdot 7$ and hence
	\[ 2^4 \geq \Psi \geq (3-1)(5-1)(7-1) = 48,\]
	which is absurd. So $k \geq 1$ and hence $m = 2$. So $f = 2^k p_1 p_2$ with $p_1, p_2 \in \{3, 5, 7\}$. If $p_1 = 5$ and $p_2 = 7$, then 
	\[ 2^4 \geq \Psi \geq 2^{k-1} (5-1) (7-1) \geq 24,\]
	which is absurd. If $p_1 = 3$ and $p_2 = 7$, then
	\[2^4 \geq \Psi \geq 2^{k-1} (3-1)(7-1) = 12 \cdot 2^{k-1},\]
	and so $k=1$ and $f = 2 \cdot 3 \cdot 7 = 42$. If $p_1 = 3$ and $p_2 = 5$, then
	\[2^4 \geq \Psi \geq 2^{k-1} (3-1)(5-1) = 8 \cdot 2^{k-1},\]
	and so $k \leq 2$ and $f \mid 2^2 \cdot 3 \cdot 5 = 60$.
	
	If $\omega(f) = 2$, then $f \leq 2^3 \cdot 7 = 56$. This completes the proof.
\end{proof}

Finally, we recall a helpful set of equivalent characterisations of $2$-elementary discriminants, which we will use in Section~\ref{subsec:2elemfields}.

\begin{proposition}[{\cite[Cor.~3.3]{AllombertBiluMadariaga15}}]\label{prop:2elemequiv}
Let $x$ be a singular modulus of discriminant $\Delta = f^2 D$ and $K = \mathbb{Q}(\sqrt{D})$. The following are equivalent:
	\begin{enumerate}
		\item The discriminant $\Delta$ is $2$-elementary.
		\item The extension $K[f]/\mathbb{Q}$ is abelian.
		\item The number field $K[f]$ is $2$-elementary.
		\item The extension $\mathbb{Q}(x) / \mathbb{Q}$ is Galois.
		\item The extension $\mathbb{Q}(x) / \mathbb{Q}$ is abelian.
	\end{enumerate}
\end{proposition}

\section{Fields generated by singular moduli}\label{sec:fields}

\subsection{Some preliminary results}

Here we recall some useful facts, which we will use repeatedly in the proof of Lemma~\ref{lem:three} and in Section~\ref{subsec:unequalare2elem}. The first is a primitive element theorem for rational linear combinations of two singular moduli. In particular, it shows that the difference $x_1 - x_2$ of distinct singular moduli $x_1, x_2$ is always a primitive element for the field $\mathbb{Q}(x_1, x_2)$.

\begin{proposition}[{\cite[Ex.~1.4 \& Th.~1.5]{BiluFayeZhu19}, \cite[Th.~4.1]{FayeRiffaut18}}]\label{prop:prim}
    Let $x_1, x_2$ be distinct singular moduli of respective discriminants $\Delta_1, \Delta_2$. Let $\epsilon \in \mathbb{Q} \setminus \{0\}$. If
    \[ \mathbb{Q}(x_1 + \epsilon x_2) \subsetneq \mathbb{Q}(x_1, x_2),\]
    then 
    \[ \left[ \mathbb{Q}\left(x_1, x_2\right) : \mathbb{Q}\left(x_1 + \epsilon x_2\right) \right] = 2\]
    and
    \begin{enumerate}
        \item either $\Delta_1 = \Delta_2$ and $\epsilon = 1$;
        \item or all of the following hold:
        \begin{enumerate}
            \item $\Delta_1 \neq \Delta_2$,
            \item $\mathbb{Q}(x_1) = \mathbb{Q}(x_2)$ and this field is a degree $2$ extension of $\mathbb{Q}$,
            \item $\epsilon = -(x_1 - x_1')/(x_2 - x_2') \notin \{\pm 1\}$, where $x_i'$ denotes the non-trivial conjugate of $x_i$ over $\mathbb{Q}$.
        \end{enumerate}
    \end{enumerate}
\end{proposition}

\begin{proposition}[{\cite[Lemma~7.1]{BiluFayeZhu19}}]\label{prop:2fields}
    Let $x_1, x_1'$ be singular moduli of discriminant $\Delta_1$ and let $x_2, x_2'$ be singular moduli of discriminant $\Delta_2$. Write $D_1, D_2$ for the respective fundamental discriminants. Suppose that $\mathbb{Q}(x_1, x_1') = \mathbb{Q}(x_2, x_2')$.
    \begin{enumerate}
        \item If $D_1 \neq D_2$, then $\mathbb{Q}(x_1) = \mathbb{Q}(x_2)$.
        \item If $D_1 = D_2$, then $K(x_1) = K(x_2)$, where $K = \mathbb{Q}(\sqrt{D_1}) = \mathbb{Q}(\sqrt{D_2})$.
    \end{enumerate}
\end{proposition}

\begin{proposition}\label{prop:1field}
    Let $x_1, x_2$ be distinct singular moduli with respective discriminants $\Delta_1, \Delta_2$ and fundamental discriminants $D_1, D_2$.
    \begin{enumerate}
        \item If $D_1 \neq D_2$ and $\mathbb{Q}(x_1) = \mathbb{Q}(x_2)$, then the discriminants $\Delta_1, \Delta_2$ are $2$-elementary and $\lvert \Delta_1 \rvert, \lvert \Delta_2 \rvert \leq 7392$.
        \item If $D_1 = D_2$ and $K(x_1) = K(x_2)$, where $K = \mathbb{Q}(\sqrt{D_1}) = \mathbb{Q}(\sqrt{D_2})$, then either $x_1, x_2 \in \mathbb{Q}$ or $\Delta_1/\Delta_2 \in \{1, 4, 1/4\}$.
    \end{enumerate}
\end{proposition}

\begin{proof}
    (1) follows from \cite[Prop.~2.8 \& Cor.~2.13]{BiluGunTron22} and Proposition~\ref{prop:Tat}. For (2), recall that the dominant singular modulus of a given discriminant is always real. Hence, every singular modulus has a real Galois conjugate over $\mathbb{Q}$. In particular, any singular modulus contained in an imaginary quadratic field must in fact be rational; hence either $K=K(x_1)=K(x_2)$ and $x_1, x_2 \in \mathbb Q$, or $K \neq K(x_1)=K(x_2)$, in which case the desired result follows from \cite[Lemma~7.2]{BiluFayeZhu19}.
\end{proof}

\subsection{Singular moduli generating \texorpdfstring{$2$}{2}-elementary fields}\label{subsec:2elemfields}

We now introduce a further bit of terminology from \cite[\S2.2.2]{BiluGunTron22}. Call a finite abelian group $G$ almost $2$-elementary if $G$ has a $2$-elementary subgroup of index $2$. A discriminant $\Delta = f^2 D$ is called almost $2$-elementary if the group $\gal(K[f] / K)$ is almost $2$-elementary, where $K = \mathbb{Q} (\sqrt{D})$. Note that every $2$-elementary discriminant is almost $2$-elementary.

In this section, we will prove the following result, which shows that if $2$ of the singular moduli occurring in the equation $x_1 - x_2 = x_3 - x_4$ have $2$-elementary discriminants, then the discriminants of the other two $2$ singular moduli are either $2$-elementary or almost $2$-elementary and equal to one another.
We will use this fact repeatedly in Sections~\ref{subsec:unequalare2elem}, \ref{Subsec:synthesis}, and \ref{sec:3dom}.

  \begin{proposition}\label{lem:all2elem}
    Let $x_1, x_2, x_3, x_4$ be pairwise distinct singular moduli with respective discriminants $\Delta_1, \Delta_2, \Delta_3, \Delta_4$. Suppose that $a_1, a_2, a_3, a_4 \in \mathbb{Q} \setminus \{0\}$ are such that
    \[ a_1 x_1 + a_2 x_2 + a_3 x_3 + a_4 x_4 \in \mathbb{Q}.\]
   If $\Delta_1, \Delta_2$ are both $2$-elementary, then:
    \begin{enumerate}
        \item either $\Delta_3, \Delta_4$ are both $2$-elementary,
        \item or $\Delta_3 = \Delta_4$ is almost $2$-elementary and $a_3 = a_4$.
    \end{enumerate}
   \end{proposition}
   
   \begin{proof}
      The result follows immediately from Lemmata~\ref{lem:compositum2elem} and \ref{lem:almost2elem} below, since
       \[ \mathbb{Q}\left(x_1 + \frac{a_2}{a_1} x_2\right) = \mathbb{Q}\left(x_3 + \frac{a_4}{a_3} x_4\right). \qedhere\]
   \end{proof}

\begin{lemma}\label{lem:compositum2elem}
Let $x_1, x_2$ be singular moduli such that their respective discriminants $\Delta_1, \Delta_2$ are both $2$-elementary. Let $\epsilon \in \mathbb{Q}$. Then the field $\mathbb{Q}(x_1 + \epsilon x_2)$ is $2$-elementary.
\end{lemma}

\begin{proof}
    For $i=1,2$, Proposition~\ref{prop:2elemequiv} implies that the extension $\mathbb{Q}(x_i) / \mathbb{Q}$ is Galois and $ \gal \left(\mathbb{Q}\left(x_i\right) / \mathbb{Q}\right) \cong (\mathbb{Z}/ 2 \mathbb{Z})^{n_i}$ for some $n_i \in \mathbb{Z}_{ \geq 0}$. Hence, the extension $\mathbb{Q}(x_1, x_2) / \mathbb{Q}$ is Galois and has Galois group isomorphic to a subgroup of $(\mathbb{Z}/ 2 \mathbb{Z})^{n_1} \times (\mathbb{Z}/ 2 \mathbb{Z})^{n_2}$. In particular, the extension  $\mathbb{Q}(x_1, x_2) / \mathbb{Q}$ is $2$-elementary, and hence the extension $\mathbb{Q}(x_1 + \epsilon x_2) / \mathbb{Q}$ is also $2$-elementary.
\end{proof}

    \begin{lemma}\label{lem:almost2elem}
    Let $x_1, x_2$ be distinct singular moduli with respective discriminants $\Delta_1, \Delta_2$. Let $\epsilon \in \mathbb{Q} \setminus \{0\}$. Suppose that the field $\mathbb{Q}(x_1 + \epsilon x_2)$ is $2$-elementary. If either $\Delta_1 \neq \Delta_2$ or $\epsilon \neq 1$, then the discriminants $\Delta_1, \Delta_2$ are $2$-elementary. If $\Delta_1 = \Delta_2$ and $\epsilon = 1$, then this discriminant is almost $2$-elementary. 
    \end{lemma}
    
    \begin{proof}
        For $i \in \{1, 2\}$, let $K_i = \mathbb{Q}(\sqrt{\Delta_i})$. By assumption, the extension $\mathbb{Q}(x_1 + \epsilon x_2) / \mathbb{Q}$ is Galois and 
        \[ \gal\left(\mathbb{Q}\left(x_1 + \epsilon x_2\right) / \mathbb{Q}\right) \cong (\mathbb{Z} / 2 \mathbb{Z})^n\]
        for some $n \in \mathbb{Z}_{\geq 0}$. For each $i \in \{1, 2\}$, we have the following diagram of field extensions:
\[
\begin{tikzcd}[row sep=2em, column sep=2em, arrows=-]
& K_i(x_1, x_2) \arrow[d]& \\
& K_i(x_1 + \epsilon x_2) \arrow[dl] \arrow[dr] &\\
\mathbb{Q}(x_1 + \epsilon x_2) \arrow[dr] & & K_i \arrow[dl] \\
& \mathbb{Q}(x_1 + \epsilon x_2) \cap K_i \arrow[d] &\\
& \mathbb{Q}. &
\end{tikzcd}
\]
By Proposition~\ref{prop:Lang}, we have that $K_i(x_1 + \epsilon x_2) / K_i$ is a Galois extension and
\[ \gal\left(K_i\left(x_1 + \epsilon x_2\right) / K_i\right) \cong \gal\left(\mathbb{Q}\left(x_1 + \epsilon x_2\right) / \mathbb{Q}\left(x_1 + \epsilon x_2\right) \cap K_i\right) \cong (\mathbb{Z} / 2 \mathbb{Z})^{m_i}\]
for some $m_i \leq n$.

Suppose that $K_i(x_1 + \epsilon x_2) = K_i(x_1, x_2)$. So the extension $K_i(x_1, x_2) / K_i$ is Galois and
\[ \gal\left(K_i\left(x_1, x_2\right) / K_i\right) \cong  (\mathbb{Z} / 2 \mathbb{Z})^{m_i}.\]
In particular, the group $\gal(K_i(x_i) / K_i)$ is isomorphic to a quotient of $(\mathbb{Z} / 2 \mathbb{Z})^{m_i}$ and hence must be $2$-elementary. Hence, the discriminant $\Delta_i$ is $2$-elementary.

So we may suppose that, for some $i \in \{1, 2\}$, 
\[ K_i(x_1 + \epsilon x_2) \subsetneq K_i(x_1, x_2).\]
This implies that
\[ \mathbb{Q}(x_1 + \epsilon x_2) \subsetneq \mathbb{Q}(x_1, x_2).\]
We thus must be in one of the two exceptional cases in Proposition~\ref{prop:prim}. If we are in exceptional case (2) of Proposition~\ref{prop:prim}, then 
\[ [ \mathbb{Q}(x_1) : \mathbb{Q} ] = [ \mathbb{Q}(x_2) : \mathbb{Q} ] = 2\]
and so certainly $\Delta_1, \Delta_2$ are $2$-elementary. We may thus assume that we are in exceptional case (1) of Proposition~\ref{prop:prim}, i.e.~$\Delta_1 = \Delta_2$ and $\epsilon = 1$.

Therefore, $K_1 = K_2$ and $K_1(x_1) = K_1(x_2)$. Proposition~\ref{prop:prim} also implies that
\[ \left[ \mathbb{Q}\left(x_1, x_2\right) : \mathbb{Q}\left(x_1 + x_2\right)\right] = 2,\]
and hence
\[ \left[K_1\left(x_1\right) : K_1\left(x_1 + x_2\right)\right] = 2.\]
Consequently,
\[ \gal\left(K_1\left(x_1\right) / K_1\right) / \gal\left(K_1(x_1)/ K_1(x_1 + x_2)\right) \cong (\mathbb{Z} / 2 \mathbb{Z})^{m_1}.\]
In particular, the group $\gal(K_1(x_1) / K_1)$ is almost $2$-elementary, and hence the discriminant $\Delta_1$ is almost $2$-elementary.
    \end{proof}

	\section{Beginning the proof of Theorem~{\ref{thm:main}}}\label{sec:easy}
	
	The ``if'' direction of Theorem~\ref{thm:main} is obvious; we will prove the ``only if''. Let $x_1, x_2, x_3, x_4$ be singular moduli. Suppose that
	\begin{align}\label{eq:main}
		x_1 - x_2 = x_3 - x_4.
		\end{align}
	We want to prove that either $(x_1, x_2) = (x_3, x_4)$ or $(x_1, x_3) = (x_2, x_4)$. Suppose then that $(x_1, x_2) \neq (x_3, x_4)$ and $(x_1, x_3) \neq (x_2, x_4)$. We will show that this leads to a contradiction.

    The proof splits into several cases.
    First, in Section~\ref{subsec:equal}, we will dispense with the case where $x_1, x_2, x_3, x_4$ are not pairwise distinct.
   Next, in Section~\ref{subsec:rational}, we handle the case where at least one of $x_1, x_2, x_3, x_4$ is rational.
    Then, in Section~\ref{subsec:general}, we set up the proof of the general case, where $x_1, x_2, x_3, x_4$ are pairwise distinct and none of them is rational.
    The proof of this general case is then split over Sections~\ref{subsec:general}, \ref{sec:2dom}, and \ref{sec:3dom}.
	
	\subsection{The non-pairwise distinct case}\label{subsec:equal}
	
	Clearly, equation \eqref{eq:main} implies that $x_1 = x_2$ if and only if $x_3 = x_4$. Similarly, $x_1 = x_3$ if and only if $x_2 = x_4$. We may thus assume that $x_1 \neq x_2$, $x_3 \neq x_4$, $x_1 \neq x_3$, and $x_2 \neq x_4$. So if $x_1, x_2, x_3, x_4$ are not pairwise distinct, then either $x_1 = x_4$ or $x_2 = x_3$, but not both. The equation \eqref{eq:main} thus reduces to either $2 x_1 = x_2 + x_3$ or $2 x_2 = x_1 + x_4$, where the three singular moduli occurring in the equation are pairwise distinct. This though contradicts the following result.
	
	\begin{lemma}\label{lem:three}
	The equation $2 x_1 = x_2 + x_3$ has no solutions with $x_1, x_2, x_3$ pairwise distinct singular moduli.
	\end{lemma}

	\begin{proof}
		Suppose that $x_1, x_2, x_3$ are pairwise distinct singular moduli such that
		\begin{align}\label{eq:3} 
		2 x_1 = x_2 + x_3.
		\end{align}
		Then $\lvert x_1 \rvert \leq \max \{\lvert x_2 \rvert, \lvert x_3 \rvert\}$. Since we may assume, by taking Galois conjugates, that $x_1$ is dominant, we must have that $\lvert \Delta_1 \rvert < \max \{\lvert \Delta_2 \rvert, \lvert \Delta_3 \rvert\}$ by Proposition~\ref{prop:dombig}. 
		
		Suppose first that $x_1 \in \mathbb Q$. Then
		\[ \mathbb{Q} (x_2 + x_3) = \mathbb{Q}.\]
		Hence, by Proposition~\ref{prop:prim}, either $x_2, x_3 \in \mathbb Q$, or $\Delta_2 = \Delta_3$ and $h(\Delta_2) = 2$. It is straightforward to check in PARI that there are no such solutions to equation \eqref{eq:3}. So we may assume that $x_1 \notin \mathbb Q$.
		
		Suppose that $x_2 \in \mathbb Q$. Then
		\[ \mathbb Q(x_3 - 2 x_1) = \mathbb Q.\]
		Hence, by Proposition~\ref{prop:prim}, we must have that 
		\[ \Delta_1 \neq \Delta_3, \mathbb Q(x_1) = \mathbb Q(x_3) \mbox{ is a degree $2$ extension of } \mathbb Q, \mbox{ and } 2=\frac{x_3 - x_3'}{x_1 - x_1'},\]
		where $x_1', x_3'$ are the non-trivial Galois conjugates of $x_1, x_3$ respectively. It is easy to check in PARI that there are no such $\Delta_1, \Delta_3$. The case where $x_3 \in \mathbb Q$ is analogous. So assume subsequently that $x_1, x_2, x_3 \notin \mathbb Q$. This implies, in particular, that $\min \{ \lvert \Delta_i \rvert : i=1, 2, 3\} \geq 15$.
		
		By Proposition~\ref{prop:prim}, we thus have that 
		\[\mathbb Q(x_2) = \mathbb Q(2 x_1 - x_3) = \mathbb Q(x_1, x_3)\]
		and 
		\[\mathbb Q(x_3) = \mathbb Q(2 x_1 - x_2) = \mathbb Q(x_1, x_2).\] 
	Hence $x_2 \in \mathbb{Q}(x_3)$ and $x_3 \in \mathbb {Q}(x_2)$, so that $\mathbb{Q}(x_2) = \mathbb{Q}(x_3)$. If $\Delta_2 \neq \Delta_3$, then Proposition~\ref{prop:prim} implies that
	\[ \mathbb Q(x_1) = \mathbb Q(x_2 + x_3) = \mathbb Q(x_2, x_3) = \mathbb{Q}(x_2) = \mathbb{Q}(x_3).\] 
	If $\Delta_2 = \Delta_3$, then, by Proposition~\ref{prop:prim} again,
		\[ \mathbb Q(x_1) = \mathbb Q(x_2 + x_3) \subset \mathbb Q(x_2, x_3) = \mathbb{Q}(x_2) = \mathbb{Q}(x_3)\]
		and 
		\[ \left[\mathbb{Q}\left(x_2\right) : \mathbb{Q}\left(x_1\right)\right] = \left[\mathbb{Q}\left(x_3\right) : \mathbb{Q}\left(x_1\right)\right] \leq 2.\] 
	
		If 
		\[ \mathbb Q (x_1) = \mathbb Q (x_2) = \mathbb Q (x_3),\]
		then equation \eqref{eq:3} and \cite[Th.~1.5]{Fowler24} imply that this field must be a degree $2$ extension of $\mathbb{Q}$. Hence, $\Delta_2 \neq \Delta_3$ in this case, else $x_2 + x_3 \in \mathbb Q$, contradicting the fact that $x_1 \notin \mathbb{Q}$. It is then easy to verify in PARI, by making use of Lemma~\ref{lem:bdsing}, that there are no such solutions to equation \eqref{eq:3}.
		
		So we must have that
		\[ \mathbb Q (x_1) \subsetneq \mathbb Q (x_2) = \mathbb Q (x_3).\]
		By equation \eqref{eq:3} and \cite[Th.~1.5]{Fowler24} again, $x_1$ must be of degree $2$ and $x_2, x_3$ of degree $4$ and $\Delta_2 = \Delta_3$. Applying a suitable Galois automorphism, we may assume from now on that $x_2$ is dominant, and so $x_3$ is not dominant. If $D_1 = D_2$, then, by \cite[Prop.~7.8]{Fowler24}, 
		\[\lvert \Delta_2 \rvert \geq 9 \lvert \Delta_1 \rvert / 4.\] 
		 Hence, by Lemma~\ref{lem:bdsing} and equation \eqref{eq:3}, we must have that
		\[ e^{ \pi \lvert \Delta_2 \rvert^{1/2}} - 2079 \leq 2 ( e^{2 \pi \lvert \Delta_2 \rvert^{1/2}/3} + 2079) + ( e^{ \pi \lvert \Delta_2 \rvert^{1/2}/2}+ 2079).\]
	This implies that $\lvert \Delta_2 \rvert \leq 8$, a contradiction. So we must have that $D_1 \neq D_2$. The possible $(\Delta_1, \Delta_2)$ may then be computed in PARI, as in \cite[Prop.~7.6]{Fowler24}. Applying the bound in Lemma~\ref{lem:bdsing} to equation \eqref{eq:3} for these $(\Delta_1, \Delta_2)$, we see that the only possibility is that $(\Delta_1, \Delta_2) = (-235, -240)$ and, in this case, $x_1$ must be dominant as well. But then $x_1 < 0$ and $x_2>0$, and so  $\lvert 2 x_1 - x_2 \rvert > \lvert x_2 \rvert > \lvert x_3 \rvert$, a contradiction.
	\end{proof}

Thus we may assume subsequently that $x_1, x_2, x_3, x_4$ are pairwise distinct.

\subsection{The case where some \texorpdfstring{$x_i \in \mathbb{Q}$}{xi in Q}}\label{subsec:rational}

Suppose $x_i \in \mathbb Q$ for some $i$. The remaining $x_j, x_k, x_l$ then satisfy 
\begin{align}\label{eq:1rat}
\epsilon_1 x_j + \epsilon_2 x_k + \epsilon_3 x_l = x_i
\end{align}
for some $\epsilon_1, \epsilon_2, \epsilon_3 \in \{\pm 1\}$, precisely one of which is $=-1$. Hence, by \cite[Theorem~1.1]{Fowler23}, we must have that either $x_j, x_k, x_l \in \mathbb Q$, or else one of $x_j, x_k, x_l$, say $x_j$, is in $\mathbb Q$ and the other two (i.e.~$x_k, x_l$) are of degree $2$ and conjugate over $\mathbb Q$ with $\epsilon_k = \epsilon_l = 1$ and $\epsilon_j = -1$. In the second case, equation \eqref{eq:1rat} rearranges to
\[x_i + x_j = x_k + x_l.\] 
It is an easy check in PARI to verify that both cases are impossible. Thus we may assume subsequently that $x_1, x_2, x_3, x_4 \notin \mathbb{Q}$. In particular, this implies that $\lvert \Delta_i \rvert \geq 15$ for every $i$.

\subsection{Beginning the general case}\label{subsec:general}

We may thus assume that
\begin{align}\label{eq:pf}
x_1 - x_2 = x_3 - x_4,
\end{align}
where $x_1, x_2, x_3, x_4$ are pairwise distinct, non-rational singular moduli.	We may further assume without loss of generality that
	\[ \lvert \Delta_1 \rvert \geq \lvert \Delta_2 \rvert, \lvert \Delta_3 \rvert, \lvert \Delta_4 \rvert \geq 15.\] 
	By taking Galois conjugates, we may also assume that $x_1$ is dominant. 
	
	The proof of Theorem~\ref{thm:main} now divides into four cases according to how many of $x_2, x_3, x_4$ are also dominant. The first two cases are straightforward and are handled in Sections~\ref{subsec:0dom} and \ref{subsec:1dom}. The remaining two cases require more work and are handled in Sections~\ref{sec:2dom} and \ref{sec:3dom}.

	\subsubsection{None of $x_2,x_3,x_4$ is dominant.}\label{subsec:0dom}
	 Then equation~\eqref{eq:pf} and Lemma \ref{lem:bdsing} imply that
	\begin{align*} 
		e^{\pi \lvert \Delta_1 \rvert^{1/2}}-2079 \leq 3(e^{\pi \lvert \Delta_1 \rvert^{1/2}/2} + 2079),
		 \end{align*}
	and this implies that $\lvert \Delta_1 \rvert \leq 8$, a contradiction.
	
	\subsubsection{Exactly one of $x_2,x_3,x_4$ is dominant.}\label{subsec:1dom}
	
	Say $i \in \{2, 3, 4\}$ is such that $x_i$ is dominant; write $x_k, x_l$ for the remaining two of $x_2, x_3, x_4$. Note that $\lvert \Delta_i \rvert < \lvert \Delta_1 \rvert$, since there is a unique dominant singular modulus of a given discriminant. So
	\begin{align*}
		\lvert x_1 \rvert - \lvert x_i \rvert &\geq  \frac{e^{\pi \lvert \Delta_1 \rvert^{1/2}}}{\lvert \Delta_1 \rvert^{1/2}}  - 2 \cdot 2079,
	\end{align*}
	by Lemmata~\ref{lem:bdsing} and \ref{lem:small}.	On the other hand, Lemma~\ref{lem:bdsing} also implies that
	\begin{align*}
		\lvert x_k \rvert + \lvert x_l \rvert \leq 2 (e^{\pi \lvert \Delta_1 \rvert^{1/2}/2} + 2079)
	\end{align*}
These bounds and equation \eqref{eq:pf} imply that $\lvert \Delta_1 \rvert \leq 10$, a contradiction.

	\section{Exactly two of \texorpdfstring{$x_2, x_3, x_4$}{x2, x3, x4} are dominant}\label{sec:2dom}
	
    This will be the most challenging case to handle.
	Write $x_i, x_j$ for the two dominant singular moduli and $x_k$ for the non-dominant singular modulus among $x_2, x_3, x_4$. Since $x_1, x_i, x_j$ are pairwise distinct and dominant, we have that $\lvert \Delta_i \rvert, \lvert \Delta_j \rvert < \lvert \Delta_1 \rvert$ and $\Delta_i \neq \Delta_j$. Dominant singular moduli are real, so $x_k \in \mathbb{R}$ also. Let $\epsilon_1, \epsilon_2, \epsilon_3 \in \{\pm 1\}$ with $\# \{n : \epsilon_n = -1\} = 1$ be such that
\begin{align}\label{eq:2dom1}
    x_1 = \epsilon_1 x_i + \epsilon_2 x_j + \epsilon_3 x_k.
\end{align} 

Our first aim is to show that we can reduce this case to the following list of cases:

\begin{enumerate}
	\item $\Delta_i, \Delta_j$ are both $2$-elementary and $\lvert \Delta_i \rvert, \lvert \Delta_j \rvert  \leq 7392$;
		\item $\Delta_i, \Delta_j$ are both $2$-elementary, $\lvert \Delta_i \rvert \leq 7392$, and $D_j = D_*$;
			\item $\Delta_i, \Delta_j$ are both $2$-elementary, $D_i = D_*$, and $\lvert \Delta_j \rvert \leq 7392$;
	\item $\Delta_i, \Delta_j$ are both $2$-elementary and $D_i = D_j = D_*$;
	\item $\Delta_i$ is $2$-elementary, $\Delta_j$ is not $2$-elementary, $\lvert \Delta_i \rvert \leq 7392$, and $D_j = D_1$;
	\item $\Delta_i$ is not $2$-elementary, $\Delta_j$ is $2$-elementary, $D_i = D_1$, and $\lvert \Delta_j \rvert \leq 7392$;
	\item $\Delta_i$ is $2$-elementary, $\Delta_j$ is not $2$-elementary, $D_i = D_*$, and $D_j = D_1$;
	\item  $\Delta_i$ is not $2$-elementary, $\Delta_j$ is $2$-elementary, $D_i = D_1$, and $D_j = D_*$.
\end{enumerate}

To do this, we first show in Section~\ref{subsec:2domSameFund} that at least one of $D_i, D_j$ must not be equal to $D_1$.
This part of the proof is straightforward.
We then prove in Section~\ref{subsec:unequalare2elem} that if $D_i \neq D_1$, then $\Delta_i$ is $2$-elementary (and analogously for $j$).
This part of the proof will take substantially more work.

In Section~\ref{Subsec:synthesis}, we deduce, by using these two facts, that we must be in one of the cases (1)--(8) above. 
We then eliminate each of these cases in turn.
This elimination relies mostly on the bounds from Section~\ref{sec:background} and on the facts about $2$-elementary discriminants established in Sections~\ref{sec:class} and \ref{sec:fields}

	\subsection{At least one of \texorpdfstring{$D_i, D_j$}{Di, Dj} is not equal to \texorpdfstring{$D_1$}{D1}.}\label{subsec:2domSameFund}

Suppose that $D_1 = D_i = D_j$. Then, by Lemmata~\ref{lem:bdsing} and \ref{lem:bdfund},
\[ \lvert x_i \rvert + \lvert x_j \rvert + \lvert x_k \rvert \leq 2(0.005 e^{\pi \lvert \Delta_1 \rvert^{1/2}} + 2079) + (e^{\pi \lvert \Delta_1 \rvert^{1/2}/2} + 2079).\]
Since, by Lemma~\ref{lem:bdsing} again,
\[ \lvert x_1 \rvert \geq e^{\pi \lvert \Delta_1 \rvert^{1/2}} - 2079,\]
we thus obtain from equation \eqref{eq:2dom1} that $\lvert \Delta_1 \rvert \leq 8$ and so $x_1 \in \mathbb Q$, a contradiction. So we may assume subsequently that at least one of $D_i, D_j$ is not equal to $D_1$.

\subsection{If \texorpdfstring{$D_i \neq D_1$}{Di neq D1}, then \texorpdfstring{$\Delta_i$}{Delta i} is \texorpdfstring{$2$}{2}-elementary.}\label{subsec:unequalare2elem}

 Suppose that $D_i \neq D_1$. Assume, towards a contradiction, that $\Delta_i$ is not $2$-elementary. Then, by Proposition~\ref{prop:auto}, there exists $\sigma_1 \in \gal(\overline{\mathbb Q} / \mathbb Q)$ such that $\sigma_1(x_i) \neq x_i$ and $\sigma_1(x_1) = x_1$. Write $x_n'$ for $\sigma_1(x_n)$. Note that $x_i'$ is not dominant.  By equation~\eqref{eq:2dom1},
\[  x_1 = \epsilon_1 x_i' + \epsilon_2 x_j' + \epsilon_3  x_k'.\]
If at least one of $x_j', x_k'$ is also not dominant, then we have reduced to one of the cases handled in Sections~\ref{subsec:0dom} and \ref{subsec:1dom}. So we may assume that $x_j', x_k'$ are both dominant. In particular, $x_j' = x_j$ and $x_k' \neq x_k$, and so
\begin{align}\label{eq:2dom2}
    x_1 = \epsilon_1 x_i' + \epsilon_2 x_j + \epsilon_3 x_k'.
\end{align} 

We obtain from equations \eqref{eq:2dom1} and \eqref{eq:2dom2} that
\begin{align}\label{eq:2dom3}
    \epsilon_1 (x_i - x_i') = -\epsilon_3 (x_k - x_k').
\end{align} 
Hence
\[ \mathbb{Q}(x_i - x_i') = \mathbb{Q} (x_k - x_k').\]
So, by Proposition~\ref{prop:prim},
\[ \mathbb{Q}(x_i, x_i') = \mathbb{Q}(x_k, x_k').\]
If $D_i \neq D_k$, then $\mathbb{Q}(x_i) = \mathbb{Q}(x_k)$ by Proposition~\ref{prop:2fields} and so $\Delta_i$ is $2$-elementary by Proposition~\ref{prop:1field}, which contradicts our assumption. So we must have that $D_i = D_k$. Then, by Propositions~\ref{prop:2fields} and \ref{prop:1field}, we have that 
\[\Delta_i / \Delta_k \in \{1, 4, 1/4\}.\] 
Suppose $\Delta_i \neq \Delta_k$. Then $\{\Delta_i, \Delta_k\} = \{\Delta, 4 \Delta\}$ for some $\Delta$. Note that $x_i, x_k'$ are dominant and $x_i', x_k$ are not dominant. Equation~\eqref{eq:2dom3} thus implies, thanks to Lemma~\ref{lem:bdsing}, that
\[ e^{2 \pi \lvert \Delta \rvert^{1/2}} -2079 \leq 3(e^{\pi \lvert \Delta \rvert^{1/2}} + 2079).\]
This however is impossible for all discriminants $\Delta$. So we must have that $\Delta_i = \Delta_k$.

 If $D_j \neq D_1$ and $\Delta_j$ is not $2$-elementary, then we could repeat the above argument with $\Delta_j$ in place of $\Delta_i$ to obtain that $\Delta_j = \Delta_k$ also. This would contradict the fact that $\Delta_i \neq \Delta_j$. Therefore, either $D_j = D_1$ or $\Delta_j$ is $2$-elementary.

Note that $x_k' = x_i$, since both are dominant and $\Delta_i = \Delta_k$. So equation \eqref{eq:2dom3} yields
\[ \epsilon_1 (x_i - x_i') = - \epsilon_3 (x_k - x_i),\]
which rearranges to
\[ x_i( \epsilon_1 - \epsilon_3) = \epsilon_1 x_i' - \epsilon_3 x_k.\]
Since $x_i$ is dominant and $x_i', x_k$ are not dominant, Proposition~\ref{prop:dombig} implies that $\epsilon_1 = \epsilon_3$. So $x_i' = x_k$ and $\epsilon_2 = -1$. Equation~\eqref{eq:2dom1} thus rearranges to
\begin{align}\label{eq:bad}
	x_1 + x_j = x_i + x_k.
\end{align}

We now split into two  cases, according to whether $\Delta_1$ is $2$-elementary. We will obtain a contradiction in each case, and thus show that $\Delta_i$ must be $2$-elementary.

\subsubsection{The case where $\Delta_1$ is not $2$-elementary}\label{subsub:Delta1not2elem}

The first case is relatively straightforward.
Suppose that $\Delta_1$ is not $2$-elementary. Since $D_1 \neq D_i = D_k$, Proposition~\ref{prop:auto} implies that there exists $\sigma_2 \in \gal(\overline{\mathbb{Q}} / \mathbb Q)$ such that $\sigma_2(x_1) \neq x_1$ but $\sigma_2(x_i) = x_i$ and $\sigma_2(x_k) = x_k$. Write $x_1'', x_j''$ for $\sigma_2(x_1), \sigma_2(x_j)$ respectively. Applying $\sigma_2$ to equation \eqref{eq:bad} and subtracting the resulting equation from equation \eqref{eq:bad}, we see that
\begin{align}\label{eq:elimbad}
	 x_1 - x_1'' = x_j'' - x_j.
	 \end{align}
In particular, $x_j'' \neq x_j$. So $x_1'', x_j''$ are not dominant. Hence, by Lemma~\ref{lem:bdsing},
\[ \lvert x_1'' \rvert + \lvert x_j'' \rvert \leq 2(e^{\pi \lvert \Delta_1 \rvert^{1/2}/2} +2079)\]
 On the other hand, by Lemmata~\ref{lem:bdsing} and \ref{lem:small},
\[ \lvert x_1 \rvert - \lvert x_j \rvert \geq \frac{e^{\pi \lvert \Delta_1 \rvert^{1/2}}}{\lvert \Delta_1 \rvert^{1/2}} - 2 \cdot 2079.\]
 These two bounds and equation \eqref{eq:elimbad} then imply that $\lvert \Delta_1 \rvert \leq 10$, which is impossible. 

\subsubsection{The case where $\Delta_1$ is $2$-elementary}\label{subsub:Delta1Listed}

The second case will require more work. 
Suppose then that $\Delta_1$ is $2$-elementary. Then, by Proposition~\ref{prop:Tat}, either $\lvert \Delta_1 \rvert \leq 7392$ or $D_1 = D_*$. 
We now split into four subcases and derive a contradiction in each.

\textit{Subcase 1.}
Suppose that $\lvert \Delta_1 \rvert \leq 7392$ and $D_j = D_1$. 
Since $\Delta_1$ is $2$-elementary and $\lvert \Delta_1 \rvert \leq 7392$, Proposition~\ref{prop:2elemcond} implies that $f_1 \leq 8$. Hence, by Lemmata~\ref{lem:bdsing} and \ref{lem:small},
\[ \lvert x_1 \rvert - \lvert x_i \rvert \geq \frac{e^{\pi \lvert \Delta_1 \rvert^{1/2}}}{8 \lvert D_1 \rvert^{1/2}} - 2 \cdot 2079.\]
Lemmata~\ref{lem:bdsing} and \ref{lem:bdfund} imply that
\[ \lvert x_j \rvert + \lvert x_k \rvert \leq e^{- \pi \lvert D_1 \rvert^{1/2}} e^{\pi \lvert \Delta_1 \rvert^{1/2}} + e^{\pi \lvert \Delta_1 \rvert^{1/2}/2} + 2 \cdot 2079.\]
These two inequalities and equation~\eqref{eq:bad} together imply that
\begin{align*} 
\frac{1}{8 \lvert D_1 \rvert^{1/2}} &\leq  e^{-\pi \lvert D_1 \rvert^{1/2}} + e^{-\pi \lvert \Delta_1 \rvert^{1/2}/2} + (4 \cdot 2079) e^{ - \pi \lvert \Delta_1 \rvert^{1/2}}\\
&\leq  2 e^{-\pi \lvert D_1 \rvert^{1/2}} +8316 e^{-2 \pi \lvert D_1 \rvert^{1/2}},
\end{align*}
where the second inequality holds since $f_1 \geq f_j + 1 \geq 2$. We obtain that $\lvert D_1 \rvert \leq 3$ and hence $\lvert \Delta_1 \rvert \leq 192$. We may then verify in PARI that for any $(\Delta_1, \Delta_i, \Delta_j, \Delta_k)$ satisfying these conditions, the inequality
\begin{equation} \label{eq:4disc1} e^{\pi \lvert \Delta_1 \rvert^{1/2}} - 2079 > (e^{\pi \lvert \Delta_i \rvert^{1/2}} + 2079) + (e^{\pi \lvert \Delta_j \rvert^{1/2}} + 2079) + (e^{\pi \lvert \Delta_k \rvert^{1/2}/2} + 2079)\end{equation}
holds. This though contradicts equation \eqref{eq:bad} by Lemma~\ref{lem:bdsing}.

\textit{Subcase 2.}
Suppose that $\lvert \Delta_1 \rvert \leq 7392$ and  $D_j \neq D_1$. 
Hence, as argued in the third paragraph of Section~\ref{subsec:unequalare2elem}, the discriminant $\Delta_j$ must be $2$-elementary. By Proposition~\ref{lem:all2elem}, this implies that $\Delta_i = \Delta_k$ and this discriminant is almost $2$-elementary. Lemma~\ref{lem:bdsing} and equation~\eqref{eq:bad} together imply that 
\begin{equation} \label{eq:4disc2} e^{\pi \lvert \Delta_1 \rvert^{1/2}} - 2079 \leq (e^{\pi \lvert \Delta_i \rvert^{1/2}} + 2079) + (e^{\pi \lvert \Delta_j \rvert^{1/2}} + 2079) + (e^{\pi \lvert \Delta_k \rvert^{1/2}/2} + 2079).\end{equation}
In PARI, we find 19 possibilities for $(\Delta_1, \Delta_i, \Delta_j, \Delta_k)$ satisfying these conditions. For each such possibility, we verify that equation~\eqref{eq:bad} does not hold with $x_1, x_i, x_j$ the dominant singular moduli of respective discriminants $\Delta_1, \Delta_i, \Delta_j$ and $x_k$ any non-dominant, real singular modulus of discriminant $\Delta_k$.

\textit{Subcase 3.} Suppose that $D_1 = D_*$ and $D_j \neq D_1$. 
As in Subcase~2, the discriminant $\Delta_j$ is $2$-elementary, since $D_j \neq D_1$. Hence $\lvert \Delta_j \rvert \leq 7392$ by Proposition~\ref{prop:Tat}. So
\[ \lvert x_j \rvert \leq e^{\pi \sqrt{7392}} + 2079\]
by Lemma~\ref{lem:bdsing}. Also,
\[ \lvert x_1 \rvert - \lvert x_i \rvert - \lvert x_k \rvert \geq \frac{e^{\pi \lvert \Delta_1 \rvert^{1/2}}}{\lvert \Delta_1 \rvert^{1/2}} - e^{\pi \lvert \Delta_1 \rvert^{1/2}/2} -3 \cdot 2079\]
by Lemmata~\ref{lem:small} and \ref{lem:bdsing}. We thus obtain from equation \eqref{eq:bad} that $\lvert \Delta_1 \rvert \leq 7638$. Hence, $\lvert D_* \rvert \leq 7638$, which contradicts Proposition~\ref{prop:Tat}.

\textit{Subcase 4.} Suppose that $D_j = D_1 = D_*$. 
Hence, by Lemmata~\ref{lem:bdsing} and \ref{lem:bdfund},
\begin{align*}
	 \lvert x_j \rvert + \lvert x_k \rvert &\leq e^{\pi \lvert \Delta_1 \rvert^{1/2}} e^{- \pi \lvert D_* \rvert^{1/2}} + e^{\pi \lvert \Delta_1 \rvert^{1/2}/2} + 2 \cdot 2079\\
	 &\leq e^{\pi \lvert \Delta_1 \rvert^{1/2}} (e^{- \pi \lvert D_* \rvert^{1/2}} + e^{- \pi \lvert \Delta_1 \rvert^{1/2}/2} + 4158 e^{- \pi \lvert \Delta_1 \rvert^{1/2}}).
	 \end{align*}
 By Lemmata~\ref{lem:bdsing} and \ref{lem:small}, we also have that
\begin{align*}
	 \lvert x_1 \rvert - \lvert x_i \rvert &\geq \frac{e^{\pi \lvert \Delta_1 \rvert^{1/2}}}{\lvert \Delta_1 \rvert^{1/2}} - 2 \cdot 2079 = e^{\pi \lvert \Delta_1 \rvert^{1/2}} \left(\frac{1}{\lvert \Delta_1 \rvert^{1/2}} -4158 e^{-\pi \lvert \Delta_1 \rvert^{1/2}}\right).
	 \end{align*}
 Thus, equation~\eqref{eq:bad} implies that
 \[\frac{1}{\lvert \Delta_1 \rvert^{1/2}} -4158 e^{-\pi \lvert \Delta_1 \rvert^{1/2}} \leq e^{- \pi \lvert D_* \rvert^{1/2}} + e^{- \pi \lvert \Delta_1 \rvert^{1/2}/2} + 4158 e^{- \pi \lvert \Delta_1 \rvert^{1/2}}.\]
 Since $\Delta_1$ is $2$-elementary, Proposition~\ref{prop:2elemcond} implies that $ \lvert D_* \rvert^{1/2} \leq \lvert \Delta_1 \rvert^{1/2} \leq 60 \lvert D_* \rvert^{1/2}$. So,
 \[ \frac{1}{60 \lvert D_* \rvert^{1/2}} \leq (2 \cdot 4158 +1) e^{- \pi \lvert D_* \rvert^{1/2}} + e^{- \pi \lvert D_* \rvert^{1/2}/2}. \]
 We thus obtain that $\lvert D_* \rvert \leq 22$, which contradicts Proposition~\ref{prop:Tat}.

\subsection{Synthesis}\label{Subsec:synthesis}

By Section~\ref{subsec:unequalare2elem}, if $\Delta_i$ is not $2$-elementary, then $D_i = D_1$. If $\Delta_i$ is $2$-elementary, then, by Proposition~\ref{prop:Tat}, either $\lvert \Delta_i \rvert \leq 7392$ or $D_i = D_*$. The same holds for $\Delta_j$. Hence, at least one of $\Delta_i, \Delta_j$ is $2$-elementary, since $D_1, D_i, D_j$ are not all equal by Section~\ref{subsec:2domSameFund}. 
We have therefore reduced to the eight cases stated at the beginning of Section~\ref{sec:2dom}, namely:

\begin{enumerate}
	\item $\Delta_i, \Delta_j$ are both $2$-elementary and $\lvert \Delta_i \rvert, \lvert \Delta_j \rvert  \leq 7392$;
		\item $\Delta_i, \Delta_j$ are both $2$-elementary, $\lvert \Delta_i \rvert \leq 7392$, and $D_j = D_*$;
			\item $\Delta_i, \Delta_j$ are both $2$-elementary, $D_i = D_*$, and $\lvert \Delta_j \rvert \leq 7392$;
	\item $\Delta_i, \Delta_j$ are both $2$-elementary and $D_i = D_j = D_*$;
	\item $\Delta_i$ is $2$-elementary, $\Delta_j$ is not $2$-elementary, $\lvert \Delta_i \rvert \leq 7392$, and $D_j = D_1$;
	\item $\Delta_i$ is not $2$-elementary, $\Delta_j$ is $2$-elementary, $D_i = D_1$, and $\lvert \Delta_j \rvert \leq 7392$;
	\item $\Delta_i$ is $2$-elementary, $\Delta_j$ is not $2$-elementary, $D_i = D_*$, and $D_j = D_1$;
	\item  $\Delta_i$ is not $2$-elementary, $\Delta_j$ is $2$-elementary, $D_i = D_1$, and $D_j = D_*$.
\end{enumerate}

We will now obtain a contradiction in each case. We do not treat cases (3), (6), and (8), because they are evidently analogous to cases (2), (5), and (7) respectively. 

Note as well that if $\Delta_i, \Delta_j$ are both $2$-elementary, then, by Proposition~\ref{lem:all2elem},  either $\Delta_1, \Delta_k$ are $2$-elementary, or: $\Delta_1 = \Delta_k$ and this discriminant is almost $2$-elementary and  
\[ x_1 + x_k = x_i + x_j.\]

\subsubsection{Case 1.}\label{subsub:2domcase1} Lemma~\ref{lem:bdsing} implies that
\[ \lvert x_i \rvert + \lvert x_j \rvert \leq 2(e^{\pi \sqrt{7392}}+2079)\]
and
\[ \lvert x_1 \rvert - \lvert x_k \rvert \geq (e^{\pi \lvert \Delta_1 \rvert^{1/2}} - 2079) - (e^{\pi \lvert \Delta_1 \rvert^{1/2}/2} + 2079).\]
We thus obtain from equation \eqref{eq:pf} that $\lvert \Delta_1 \rvert \leq 7429$.

Lemma~\ref{lem:bdsing} and equation~\eqref{eq:pf} imply that 
\[ e^{\pi \lvert \Delta_1 \rvert^{1/2}} - 2079 \leq (e^{\pi \lvert \Delta_i \rvert^{1/2}} + 2079) + (e^{\pi \lvert \Delta_j \rvert^{1/2}} + 2079) + (e^{\pi \lvert \Delta_k \rvert^{1/2}/2} + 2079).\]
We may find in PARI all the possibilities for $(\Delta_1, \Delta_i, \Delta_j, \Delta_k)$ satisfying these conditions with $\lvert \Delta_1 \rvert \leq 7429$. For each such $(\Delta_1, \Delta_i, \Delta_j, \Delta_k)$, we may then verify that equation \eqref{eq:pf} does not hold whenever $x_1, x_i, x_j$ are the dominant singular moduli of respective discriminants $\Delta_1, \Delta_i, \Delta_j$ and $x_k$ is any non-dominant, real singular modulus of discriminant $\Delta_k$.

\subsubsection{Case 2.} By Lemmata~\ref{lem:bdsing} and \ref{lem:small},
\[ \lvert x_1 \rvert - \lvert x_j  \rvert \geq \frac{e^{\pi \lvert \Delta_1 \rvert^{1/2}}}{\lvert \Delta_1 \rvert^{1/2}} - 2 \cdot 2079\]
and, by Lemma~\ref{lem:bdsing},
\[ \lvert x_i \rvert + \lvert x_k \rvert  \leq e^{\pi \sqrt{7392}} + e^{\pi \lvert \Delta_1 \rvert^{1/2}/2} +2 \cdot 2079.\]
We obtain from equation \eqref{eq:pf} that $\lvert \Delta_1 \rvert \leq 7638$. In particular, $\lvert D_* \rvert \leq \lvert \Delta_j \rvert \leq 7638$, which contradicts Proposition~\ref{prop:Tat}.

\subsubsection{Case 4.} Without loss of generality, assume that $\lvert \Delta_j \rvert < \lvert \Delta_i \rvert$. So, by Lemmata~\ref{lem:small}, \ref{lem:bdfund}, and \ref{lem:bdsing},
\[ \lvert x_1 \rvert - \lvert x_i \rvert - \lvert x_j \rvert - \lvert x_k \rvert \geq \frac{e^{\pi \lvert \Delta_1 \rvert^{1/2}}}{\lvert \Delta_1 \rvert^{1/2}} - e^{-\pi \lvert D_* \rvert^{1/2}} e^{\pi \lvert \Delta_i \rvert^{1/2}} - e^{\pi \lvert \Delta_1 \rvert^{1/2}/2} -4 \cdot 2079.\]
Hence, equation \eqref{eq:pf} implies that
\[ e^{-\pi \lvert D_* \rvert^{1/2}} e^{\pi \lvert \Delta_i \rvert^{1/2}} \geq e^{\pi \lvert \Delta_1 \rvert^{1/2}} \left(\frac{1}{\lvert \Delta_1 \rvert^{1/2}} - e^{-\pi \lvert \Delta_1 \rvert^{1/2}/2} - 8316 e^{-\pi \lvert \Delta_1 \rvert^{1/2}}\right).\]

Note that
\[ \frac{1}{\lvert \Delta_1 \rvert^{1/2}} - e^{-\pi \lvert \Delta_1 \rvert^{1/2}/2} - 8316 e^{-\pi \lvert \Delta_1 \rvert^{1/2}} \geq \frac{1}{2 \lvert \Delta_1 \rvert^{1/2}}\]
if $\lvert \Delta_1 \rvert \geq 13$. Since $\lvert \Delta_1 \rvert \geq 15$, we thus must have that
\begin{align*}
e^{-\pi \lvert D_* \rvert^{1/2}} e^{\pi \lvert \Delta_i \rvert^{1/2}} \geq  \frac{e^{\pi \lvert \Delta_1 \rvert^{1/2}}}{2 \lvert \Delta_1 \rvert^{1/2}} .
\end{align*}
The function $g(t) = \exp(\pi \sqrt{t})/2 \sqrt{t}$ is increasing for $t>1/\pi^2$. So
\begin{align*}
	e^{-\pi \lvert D_* \rvert^{1/2}} e^{\pi \lvert \Delta_i \rvert^{1/2}} \geq  \frac{e^{\pi \lvert \Delta_i \rvert^{1/2}}}{2 \lvert \Delta_i \rvert^{1/2}}.
\end{align*}
 Since $\Delta_i$ is $2$-elementary, $f_i \leq 60$ by Proposition~\ref{prop:2elemcond}, and so $\lvert \Delta_i \rvert^{1/2} \leq 60 \lvert D_* \rvert^{1/2}$. So
\[	e^{-\pi \lvert D_* \rvert^{1/2}} \geq  \frac{1}{2 \lvert \Delta_i \rvert^{1/2}} \geq \frac{1}{2 \cdot 60 \lvert D_* \rvert^{1/2}},\]
which is impossible, since $\lvert D_* \rvert \geq 3$.

\subsubsection{Case 5.}\label{subsub:2domcase2} By Lemma~\ref{lem:bdsing}, we have that
\[ \lvert x_1 \rvert > e^{\pi \lvert \Delta_1 \rvert^{1/2}} - 2079\]
and also, thanks to Lemma~\ref{lem:bdfund},
\[ \lvert x_i \rvert + \lvert x_j \rvert + \lvert x_k \rvert \leq  e^{\pi \sqrt{7392}} + 0.005 e^{\pi \lvert \Delta_1 \rvert^{1/2}} + e^{\pi \lvert \Delta_1 \rvert^{1/2}/2} + 3 \cdot 2079.\]
We thereby obtain from equation \eqref{eq:pf} that $\lvert \Delta_1 \rvert \leq 7392$. 

Equation~\eqref{eq:pf} and Lemmata~\ref{lem:bdsing} and \ref{lem:bdfund} imply that
\[ e^{\pi \lvert \Delta_1 \rvert^{1/2}} - 2079 < (e^{\pi \lvert \Delta_i \rvert^{1/2}} + 2079) + (0.005e^{\pi \lvert \Delta_1 \rvert^{1/2}} + 2079) + (e^{\pi \lvert \Delta_k \rvert^{1/2} /2} + 2079).\]
Thus,
\[ 0.995 e^{\pi \lvert \Delta_1 \rvert^{1/2}} < e^{\pi \lvert \Delta_i \rvert^{1/2}} + e^{\pi \lvert \Delta_k \rvert^{1/2} /2} + 8316.\]
Since $\lvert \Delta_k \rvert \leq \lvert \Delta_1 \rvert$, we have that
\[ e^{\pi \lvert \Delta_1 \rvert^{1/2}} (0.995 - e^{- \pi \lvert \Delta_1 \rvert^{1/2} /2} -8316 e^{- \pi \lvert \Delta_1 \rvert^{1/2}}) < e^{\pi \lvert \Delta_i \rvert^{1/2}}.\]

The discriminants $\Delta_1, \Delta_i, \Delta_j$ are pairwise distinct and are all $\geq 15$ in absolute value. So $\lvert \Delta_1 \rvert \geq 19$, since discriminants are $\equiv 0, 1 \bmod 4$. So
\[ 0.995 - e^{- \pi \lvert \Delta_1 \rvert^{1/2} /2} -8316 e^{- \pi \lvert \Delta_1 \rvert^{1/2}} > 0.995 - e^{- \pi \sqrt{19}/2} -8316 e^{-\pi \sqrt{19}} > 0.9845.\]
Hence,
\[ 0.9845 e^{\pi \lvert \Delta_1 \rvert^{1/2}} < e^{\pi \lvert \Delta_i \rvert^{1/2}}.\]
Thus,
\[ \lvert \Delta_1 \rvert^{1/2} - \lvert \Delta_i \rvert^{1/2} < \frac{- \log(0.9845)}{\pi} < 0.005.\]
We may therefore obtain that
\[ \lvert \Delta_1 \rvert < \lvert \Delta_i \rvert + 0.01 \lvert \Delta_i \rvert^{1/2} + 0.000025 < \lvert \Delta_i \rvert + 1,\]
since $\lvert \Delta_i \rvert^{1/2} \leq \sqrt{7392} < 86$. This though contradicts the fact that $\lvert \Delta_1 \rvert \geq \lvert \Delta_i \rvert + 1$.

\subsubsection{Case 7.} Since $\Delta_1 \neq \Delta_j$, Proposition~\ref{lem:all2elem} implies that $\Delta_k$ is not $2$-elementary. Note also that $D_1 \neq D_*$, since $D_1, D_i, D_j$ are not all equal. So, if $\Delta_1$ is $2$-elementary, then $\lvert \Delta_1 \rvert \leq 7392$ by Proposition~\ref{prop:Tat}. But then $\lvert D_* \rvert = \lvert D_i \rvert \leq 7392$, which contradicts Proposition~\ref{prop:Tat}. So we may assume that $\Delta_1$ is not $2$-elementary. 

Suppose that $D_k \neq D_1$. 
Since $\Delta_k$ is not $2$-elementary, there exists, by Proposition~\ref{prop:auto}, $\sigma_1 \in \gal(\overline{\mathbb{Q}} / \mathbb Q)$ such that $\sigma_1(x_k) \neq x_k$ but $\sigma_1(x_1) = x_1$ and $\sigma_1(x_j) = x_j$. Writing $x_n'$ for $\sigma_1(x_n)$, we obtain from equation~\eqref{eq:2dom1} that
\[ \epsilon_1 (x_i - x_i') = -\epsilon_3 (x_k - x_k').\]
We may now obtain a contradiction by arguing as in the paragraph following equation \eqref{eq:2dom2}. If $D_i \neq D_k$, then $\Delta_k$ is $2$-elementary, a contradiction. If $D_i = D_k$, then $\Delta_i / \Delta_k \in \{1, 4, 1/4\}$. If $\Delta_i \neq \Delta_k$, then we obtain a bound $\max \{ \lvert \Delta_i \rvert, \lvert \Delta_k \rvert\} \leq 11$, which is impossible. So we must have that $\Delta_i = \Delta_k$, and this contradicts the fact that $\Delta_k$ is not $2$-elementary. 

So $D_k = D_1 = D_j$. By Proposition~\ref{prop:auto}, there exists $\sigma_2 \in \gal(\overline{\mathbb Q} / \mathbb{Q})$ such that $\sigma_2(x_1) \neq x_1$ but $\sigma_2(x_i) = x_i$. We may thus obtain from equation~\eqref{eq:2dom1} that
\begin{align}\label{eq:elim1}
	x_1 - x_1'' = \epsilon_2 (x_j - x_j'') + \epsilon_3 (x_k - x_k''),
\end{align}
where $x_n''$ denotes $\sigma_2(x_n)$. In particular, $x_1''$ is not dominant. Lemma~\ref{lem:bdsing} implies that
\begin{align}\label{eq:x1bd}
    \lvert x_1 \rvert > e^{\pi \lvert \Delta_1 \rvert^{1/2}} - 2079.
\end{align} 
Suppose that $\Delta_k \neq \Delta_1$. Then, by Lemma~\ref{lem:bdfund}, we also have that
\[ \lvert x_1'' \rvert + \lvert x_j \rvert + \lvert x_j'' \rvert + \lvert x_k \rvert + \lvert x_k'' \rvert \leq (e^{\pi \lvert \Delta_1 \rvert^{1/2}/2} + 2079) + 4 (0.005e^{\pi \lvert \Delta_1 \rvert^{1/2}} + 2079).\]
These bounds and equation \eqref{eq:elim1} imply that $\lvert \Delta_1 \rvert \leq 9$, which is impossible. 

So $\Delta_k = \Delta_1$. By Proposition~\ref{prop:auto}, there exists $\sigma_3 \in \gal(\overline{\mathbb Q} / \mathbb{Q})$ such that $\sigma_3(x_j) \neq x_j$ but $\sigma_3(x_i) = x_i$. Equation~\eqref{eq:2dom1} thus yields that
\begin{align}\label{eq:elim}
	x_1 - x_1''' = \epsilon_2 (x_j - x_j''') + \epsilon_3 (x_k - x_k'''),
\end{align}
where $x_n'''$ denotes $\sigma_3(x_n)$. Suppose that $x_k = x_k'''$. Then $x_1 \neq x_1'''$, since $x_j \neq x_j'''$. So $x_1'''$ is not dominant. Lemmata~\ref{lem:bdsing} and \ref{lem:bdfund} imply that
\[ \lvert x_1 \rvert \geq e^{\pi \lvert \Delta_1 \rvert^{1/2}} - 2079\]
and
\[ \lvert x_1''' \rvert + \lvert x_j \rvert + \lvert x_j''' \rvert \leq (e^{\pi \lvert \Delta_1 \rvert^{1/2}/2} + 2079) + 2 (0.005 e^{\pi \lvert \Delta_1 \rvert^{1/2}} + 2079).\]
These bounds and equation~\eqref{eq:elim} then imply that  $\lvert \Delta_1 \rvert \leq 8$, which is impossible. So $x_k \neq x_k'''$. If $x_1 = x_1'''$, then we may argue similarly (applying first a suitable automorphism to \eqref{eq:elim} to make $x_k$ dominant) and obtain a contradiction. So $x_1 \neq x_1'''$ as well. In particular, $x_1'''$ is not dominant.

If $x_k'''$ is not dominant, then Lemmata~\ref{lem:bdsing} and \ref{lem:bdfund} imply that
\[ \lvert x_1''' \rvert + \lvert x_j \rvert + \lvert x_j''' \rvert + \lvert x_k \rvert + \lvert x_k''' \rvert \leq 3(e^{\pi \lvert \Delta_1 \rvert^{1/2}/2} + 2079) + 2 (0.005 e^{\pi \lvert \Delta_1 \rvert^{1/2}} + 2079).\]
Equation~\eqref{eq:elim} and the bound \eqref{eq:x1bd} then imply that $\lvert \Delta_1 \rvert \leq 9$, which is impossible.

So assume that $x_k'''$ is dominant, i.e.~$x_k''' = x_1$. Lemmata~\ref{lem:bdsing} and \ref{lem:bdfund} then yield that
\[ \lvert x_1''' \rvert + \lvert x_j \rvert + \lvert x_j''' \rvert + \lvert x_k \rvert \leq 2(e^{\pi \lvert \Delta_1 \rvert^{1/2}/2} + 2079) + 2 (0.005 e^{\pi \lvert \Delta_1 \rvert^{1/2}} + 2079).\]
If $\epsilon_3 = 1$, then 
\[ \lvert x_1 + \epsilon_3 x_k'' \rvert \geq 2 (e^{\pi \lvert \Delta_1 \rvert^{1/2}} - 2079).\]
And so equation~\eqref{eq:elim} implies that $\lvert \Delta_1 \rvert \leq 7$, a contradiction. If $\epsilon_3 = -1$, then $\epsilon_2 = 1$ and equation~\eqref{eq:elim} rearranges to 
\begin{align}\label{eq:badbad}
x_k - x_1''' = x_j - x_j'''.
\end{align}
Since $x_j \neq x_j'''$, we have that $x_k \neq x_1'''$. Taking suitable Galois conjugates, we may now assume that $x_k$ is dominant, and so $x_1'''$ is not dominant. Thus, by Lemma~\ref{lem:bdsing},
\[ \lvert x_k \rvert - \lvert x_1''' \rvert \geq (e^{\pi \lvert \Delta_1 \rvert^{1/2}} - 2079) - (e^{\pi \lvert \Delta_1 \rvert^{1/2}/2} + 2079),\]
while, by Lemma~\ref{lem:bdfund},
\[ \lvert x_j \rvert + \lvert x_j''' \rvert \leq 2(0.005 e^{\pi \lvert \Delta_1 \rvert^{1/2}} + 2079).\]
These bounds and equation \eqref{eq:badbad} together imply that $\lvert \Delta_1 \rvert \leq 8$, a contradiction.

	\section{All of \texorpdfstring{$x_2, x_3, x_4$}{x2, x3, x4} are dominant}\label{sec:3dom}
	
In this case, $\Delta_1, \Delta_2, \Delta_3, \Delta_4$ are pairwise distinct, since $x_1, x_2, x_3, x_4$ are dominant and pairwise distinct. 
	
	\subsection{Equal fundamental discriminants}\label{subsec:3domSameFund}
	
	If $D_1 = D_2 = D_3 = D_4$, then Lemmata~\ref{lem:bdsing} and \ref{lem:bdfund} and equation \eqref{eq:pf} together imply that
	\[ e^{\pi \lvert \Delta_1 \rvert^{1/2}} - 2079 \leq 3\,(0.005 e^{\pi \lvert \Delta_1 \rvert^{1/2}} + 2079),\]
	and hence $\lvert \Delta_1 \rvert \leq 8$, a contradiction. We may thus assume that the fundamental discriminants $D_1, D_2, D_3, D_4$ are not all equal.
	
	\subsection{An inductive argument}\label{subsec:ind}
	
	Let $i \in \{2, 3, 4\}$ be such that $D_i \neq D_1$. Suppose that $\Delta_i$ is not $2$-elementary. Then, by Proposition~\ref{prop:auto}, there exists $\sigma \in \gal(\overline{\mathbb Q} / \mathbb Q )$ such that $\sigma(x_i) \neq x_i$ and $\sigma(x_1) = x_1$. In particular, $\sigma(x_i)$ is not dominant, and so applying $\sigma$ to equation \eqref{eq:pf} reduces us to one of the previously proved cases where at most two of $x_2, x_3, x_4$ are dominant. 
    
    We may thus assume subsequently that, for every $i \in \{ 2, 3, 4\}$, either $D_i = D_1$ or $\Delta_i$ is $2$-elementary. In particular, at least one of $\Delta_2, \Delta_3, \Delta_4$ is $2$-elementary, since $D_1, D_2, D_3, D_4$ are not all equal. 
    To finish the proof of Theorem~\ref{thm:main}, we distinguish two final cases, according to whether or not $\Delta_1$ is $2$-elementary.
	
	\subsection{The case where \texorpdfstring{$\Delta_1$}{Delta 1} is not \texorpdfstring{$2$}{2}-elementary}
	
	 Suppose $\Delta_1$ is not $2$-elementary. Then, by Proposition~\ref{lem:all2elem}, at most one of $\Delta_2, \Delta_3, \Delta_4$ is $2$-elementary, since the four discriminants are pairwise distinct. Therefore, there exists a unique $i \in \{2, 3, 4\}$ such that $D_i \neq D_1$, since $\Delta_i$ is $2$-elementary for any such $i$ by Section~\ref{subsec:ind}. For the remaining $j, k \in \{2, 3, 4\}$, we have that $D_j = D_k = D_1$ and hence, by Lemma~\ref{lem:bdfund},
	\begin{align}\label{eq:all4dom}
	\lvert x \rvert \leq 0.005 e^{\pi \lvert \Delta_1 \rvert^{1/2}} + 2079
	\end{align}
	for every singular modulus $x$ of discriminant $\Delta_j$ or $\Delta_k$.
	
	Since $\Delta_1$ is not $2$-elementary, by Proposition~\ref{prop:auto} there exists $\sigma \in \gal(\overline{\mathbb Q}/\mathbb{Q})$ such that $\sigma(x_1) \neq x_1$ and $\sigma(x_i) = x_i$. In particular, $\sigma(x_1)$ is not dominant. Write $x_n'$ for $\sigma(x_n)$. There are $\epsilon_1, \epsilon_2, \epsilon_3 \in \{\pm 1\}$ such that
	\[ x_1 = \epsilon_1 x_i + \epsilon_2 x_j + \epsilon_3 x_k\]
	and
	\[ x_1' = \epsilon_1 x_i + \epsilon_2 x_j' + \epsilon_3 x_k'.\]
	Hence,
	\[ x_1 =  x_1' + \epsilon_2 (x_j - x_j') + \epsilon_3 (x_k - x_k').\]
	This equation,  Lemma~\ref{lem:bdsing}, and equation \eqref{eq:all4dom} together imply that
	\[ e^{\pi \lvert \Delta_1 \rvert^{1/2}} - 2079 \leq (e^{\pi \lvert \Delta_1 \rvert^{1/2}/2} + 2079) + 4 (0.005 e^{\pi \lvert \Delta_1 \rvert^{1/2}} + 2079).\]
	From this we obtain that $\lvert \Delta_1 \rvert \leq 9$, a contradiction.
	 
	 \subsection{The case where \texorpdfstring{$\Delta_1$}{Delta 1} is \texorpdfstring{$2$}{2}-elementary}\label{subsec:D12elem}
	
	So we may assume subsequently that $\Delta_1$ is $2$-elementary. Since the fundamental discriminants $D_1, D_2, D_3, D_4$ are not all equal, there exists $i \in \{2, 3, 4\}$ such that $D_i \neq D_1$. Then $\Delta_i$ is $2$-elementary by Section~\ref{subsec:ind}. Therefore, Proposition~\ref{lem:all2elem} implies that all of $\Delta_1, \Delta_2, \Delta_3, \Delta_4$ must be $2$-elementary, since they are pairwise distinct.
	
	Suppose that $D_1 = D_*$. Then, for each $j \geq 2$, either
	\[ D_j = D_1 \mbox{ and } \lvert x_j \rvert \leq 0.005 e^{\pi \lvert \Delta_1 \rvert^{1/2}} + 2079\]
	by Lemmata~\ref{lem:bdsing} and \ref{lem:bdfund}, or
	\[ D_j \neq D_1 \mbox{ and } \lvert \Delta_j \rvert \leq 7392\]
	by Proposition~\ref{prop:Tat}. These bounds, Lemma~\ref{lem:bdsing}, and equation \eqref{eq:pf} then imply that $\lvert \Delta_1 \rvert \leq 7452$ and thus $\lvert D_* \rvert \leq 7452$, which contradicts Proposition~\ref{prop:Tat}. 
	
	So we may assume that $D_1 \neq D_*$. Hence, $\lvert \Delta_1 \rvert \leq 7392$ by Proposition~\ref{prop:Tat}. Further, Lemma~\ref{lem:bdsing} and equation~\eqref{eq:pf} imply that
	\begin{equation}\label{eq:4disc3} e^{\pi \lvert \Delta_1 \rvert^{1/2}} - 2079 \leq (e^{\pi \lvert \Delta_2 \rvert^{1/2}} +2079) + (e^{\pi \lvert \Delta_3 \rvert^{1/2}} + 2079) + (e^{\pi \lvert \Delta_4 \rvert^{1/2}} + 2079)\end{equation}
	We may then find in PARI all the possibile $(\Delta_1, \Delta_2, \Delta_3, \Delta_4)$ satisfying the above conditions and, for each possibility, verify that equation \eqref{eq:pf} does not hold for the corresponding dominant singular moduli.
    This completes the proof of Theorem~\ref{thm:main}.

	\bigskip
	\textbf{Acknowledgments.} This work was supported by ANR project JINVARIANT. The authors would like to thank the referee for their helpful comments. On behalf of all authors, the corresponding author states that there is no conflict of interest. 

\smallskip

    \textbf{Supporting data.} All computations in this paper were performed using the open-source computer algebra system PARI/GP \cite{PARI2}; the scripts are available from: \url{https://github.com/guyfowler/distinct_differences}.


\begin{thebibliography}{99}
	
	\bibitem{AllombertBiluMadariaga15}
	B.~Allombert, Yu. Bilu, and A.~Pizarro-Madariaga, \emph{C{M}-points on straight
		lines}, Analytic number theory, Springer, Cham, 2015, pp.~1--18.
		
		\bibitem{Andre98}
Y.~Andr\'{e}, \emph{Finitude des couples d'invariants modulaires singuliers sur
  une courbe alg\'{e}brique plane non modulaire}, J. Reine Angew. Math.
  \textbf{505} (1998), 203--208.
  
  \bibitem{AslanyanEterovicFowler23}
V.~Aslanyan, S.~Eterović, and G.~Fowler, \emph{Multiplicative relations among
  differences of singular moduli},  \texttt{ar{X}iv:2308.12244v2}, to appear in Ann. Sc. Norm. Super. Pisa Cl. Sci.
	
	\bibitem{BiluFayeZhu19}
	Yu.~Bilu, B.~Faye, and H.~Zhu, \emph{Separating singular moduli and the
		primitive element problem}, Q. J. Math. \textbf{71} (2020), no.~4,
	1253--1280.
	
	\bibitem{BiluGunTron22}
  Yu.~Bilu, S.~Gun and E.~Tron, Effective multiplicative independence of three singular moduli, Algebra \& Number Theory {\bf 20} (2026), no.~6, 1073--1123.
  
  \bibitem{BiluHabeggerKuehne20} Yu.~Bilu, P.~Habegger, and L.~Kühne, \emph{No singular modulus is a unit}, Int. Math. Res. Not. IMRN (2020), no. 24, 10005--10041.
  
  \bibitem{BiluKuhne20}
Yu. Bilu and L.~K\"{u}hne, \emph{Linear {E}quations in {S}ingular {M}oduli},
  Int. Math. Res. Not. IMRN (2020), no.~21, 7617--7643.
  
  \bibitem{BiluLucaMasser17}
Yu.~Bilu, F.~Luca, and D.~Masser, \emph{Collinear {CM}-points}, Algebra \&
  Number Theory \textbf{11} (2017), no.~5, 1047--1087.
	
	\bibitem{BiluMasserZannier13}
	Yu.~Bilu, D.~Masser, and U.~Zannier, \emph{An effective ``theorem of
		{A}ndr\'{e}'' for {CM}-points on a plane curve}, Math. Proc. Cambridge
	Philos. Soc. \textbf{154} (2013), no.~1, 145--152.
	
	\bibitem{Binyamini19}
G.~Binyamini, \emph{Some effective estimates for {A}ndr\'{e}--{O}ort in
  {$Y(1)^n$}}, J. Reine Angew. Math. \textbf{767} (2020), 17--35, with an
  appendix by E. Kowalski.
  
  \bibitem{Breuer01}
F.~Breuer, \emph{Heights of {CM} points on complex affine curves}, Ramanujan J.
  \textbf{5} (2001), no.~3, 311--317.
	
	\bibitem{Campagna20}
	F.~Campagna, \emph{On singular moduli that are $S$-units}, Manuscripta Math. \textbf{166} (2021), 73--90.
	
	
	\bibitem{Cohn94}
	H.~Cohn, \emph{Introduction to the construction of class fields}, Dover
	Publications, Inc., New York, 1994, Corrected reprint of the 1985 original.
	
	\bibitem{Cox89}
D.~Cox, \emph{Primes of the form {$x^2 + ny^2$}}, John Wiley \& Sons, Inc., New
  York, 1989.
  
  \bibitem{Edixhoven98}
B.~Edixhoven, \emph{Special points on the product of two modular curves},
  Compositio Math. \textbf{114} (1998), no.~3, 315--328.
	
	\bibitem{FayeRiffaut18}
	B.~Faye and A.~Riffaut, \emph{Fields generated by sums and products of singular
		moduli}, J. Number Theory \textbf{192} (2018), 37--46.
	
	\bibitem{Fowler23}
	G.~Fowler, \emph{Equations in three singular moduli: the equal exponent case},
	J. Number Theory \textbf{243} (2023), 256--297.
	
	\bibitem{Fowler24}
	G.~Fowler, \emph{Some uniform effective results on Andr\'e-Oort for sums of powers in $\mathbb{C}^n$}, Math. Proc. Cambridge Philos. Soc. {\bf 180} (2026), no.~3, 607--641.
		
		\bibitem{FuZhao24}
Y.~Fu and R.~Zhao, \emph{Algebraic independence of special points on {S}himura
  varieties}, preprint, \texttt{ar{X}iv:2405.14084v2} (2024).
		
		\bibitem{GrossZagier85}
B.~Gross and D.~Zagier, \emph{On singular moduli}, J. Reine Angew. Math.
  \textbf{355} (1985), 191--220.
  
  \bibitem{HerreroMenaresRiveraletelier24}
S. Herrero, R. Menares, and J. Rivera-Letelier, \emph{There are at most finitely many singular moduli that are $S$-units}, Compos. Math. \textbf{160}.4 (2024), 732--770.
  
		
		\bibitem{Kuhne13}
L.~K{\"{u}}hne, \emph{An effective result of {A}ndr\'{e}--{O}ort type {II}},
  Acta Arith. \textbf{161} (2013), no.~1, 1--19.
	
	\bibitem{Kuhne21}
	L.~K\"{u}hne, \emph{Intersection of class fields}, Acta Arith. \textbf{198}
	(2021), no.~2, 109--127.
	
	\bibitem{Lang02}
	S.~Lang, \emph{Algebra}, third ed., Graduate Texts in Mathematics, vol. 211,
	Springer-Verlag, New York, 2002.
	
	   \bibitem{LauterViray15} K.~Lauter and B.~Viray, \emph{On singular moduli for arbitrary discriminants}, Int. Math. Res. Not. IMRN (2015), no.~19, 9206--9250.
	
	\bibitem{Li21}
Y.~Li, \emph{Singular units and isogenies between {CM} elliptic curves},
  Compos. Math. \textbf{157} (2021), no.~5, 1022--1035.
	
	    \bibitem{PARI2}
    The PARI~Group, PARI/GP version \texttt{2.18.0}, Univ. Bordeaux, 2024,
    \url{http://pari.math.u-bordeaux.fr/}.
    
    \bibitem{Pazuki19}
    F.~Pazuki, \emph{Modular invariants and isogenies}, Int. J. Number Theory \textbf{15}.3 (2019), 569--584.
	
	\bibitem{Pila11}
J.~Pila, \emph{O-minimality and the {A}ndr\'e--{O}ort conjecture for {$\mathbb{
  C}^n$}}, Ann. of Math. (2) \textbf{173} (2011), no.~3, 1779--1840.
  
  \bibitem{Pila14a}
J.~Pila, \emph{Special point problems with elliptic modular surfaces},
  Mathematika \textbf{60} (2014), no.~1, 1--31.
  
  \bibitem{PilaTsimerman17}
J.~Pila and J.~Tsimerman, \emph{Multiplicative relations among singular
  moduli}, Ann. Sc. Norm. Super. Pisa Cl. Sci. (5) \textbf{17} (2017), no.~4,
  1357--1382.

\bibitem{Scanlon04}
T.~Scanlon, \emph{Automatic uniformity}, Int. Math. Res. Not. (2004), no.~62,
  3317--3326.
	
	\bibitem{Siegel35}
C.~Siegel, \emph{{\"U}ber die {Classenzahl} quadratischer {Zahlk{\"o}rper}},
  Acta Arith. \textbf{1} (1935), 83--86.
	
	\bibitem{Tatuzawa51}
T.~Tatuzawa, \emph{On a theorem of {S}iegel}, Jpn. J. Math. \textbf{21} (1951),
  163--178.
	
\end{thebibliography}
\end{document}